\documentclass[11pt,preprint]{imsart}

\usepackage{amsthm,amssymb,amsmath,amsfonts,textcomp} 
\usepackage[normalem]{ulem} 
\usepackage{graphicx,caption,subcaption,stfloats,tikz,epstopdf} 
\usepackage{array,enumerate,url,verbatim,xcolor,bbm,bm}

\RequirePackage[round,colon,authoryear]{natbib}
\RequirePackage[colorlinks,citecolor=blue,urlcolor=blue]{hyperref}
\usepackage{color,fancyvrb,mathtools, undertilde}
\newtheorem{theorem}{Theorem}[section]
\newtheorem*{theorem*}{Theorem}

\newtheorem{lemma}{Lemma}
\newtheorem*{lemma*}{Lemma}

\newtheorem{remark}{Remark}[section]
\newtheorem{example}{Example}[section]

\providecommand{\abs}[1]{\left\lvert#1\right\rvert}
\providecommand{\norm}[1]{\left\lVert#1\right\rVert}

\renewcommand{\hat}{\widehat}

\newcommand{\bfm}[1]{\ensuremath{\mathbf{#1}}}

   \def\bB{\bfm B}   
      \def\cC{{\cal  C}}
   \def\bD{\bfm D}   \def\cD{{\cal  D}}
     \def\EE{\mathbb{E}} 
     
   \def\bG{\bfm G}   
      
   \def\bI{\bfm I}   
   \def\bJ{\bfm J}   
   \def\bK{\bfm K}   
\def\hbbeta{\bfm l}      
   \def\bM{\bfm M}

     \def\PP{\mathbb{P}} 
      
   \def\bR{\bfm R}  \def\RR{\mathbb{R}}

\def\bu{\bfm u}      
\def\bv{\bfm v}      
      
\def\bx{\bfm x}   \def\bX{\bfm X}   
\def\by{\bfm y}      
\def\bz{\bfm z}      

\def\polylog{{\rm PolyLog}}

\def\LO{{\rm LO}}
\def\OO{{\rm OO}}

\def\dphi{\dot{\phi}}
\def \dl{\dot{\ell}}
\def \ddl{\ddot{\ell}}
\def \prox{\mathbf{prox}}

\newcommand{\bfsym}[1]{\ensuremath{\boldsymbol{#1}}}

\def\bbeta{\bfsym \beta}

           \def\bTheta {\bfsym {\Theta}}
          
             \def\bSigma{\bfsym \Sigma}

\def\bxi{\bfsym {\xi}} \def\bXi{\bfsym {\Xi}}

\def\bPi{\bfsym{\Pi}}

                \def\hbbeta{\hat{\bfsym \beta}}

\DeclareMathOperator{\argmin}{argmin}

\DeclareMathOperator{\diag}{diag}

\DeclareMathOperator{\var}{var}

\DeclareMathOperator{\tr}{tr}

\newcommand\numberthis{\addtocounter{equation}{1}\tag{\theequation}}

\begin{document}
\begin{frontmatter}

\title{Theoretical Analysis of Leave-one-out Cross Validation for Non-differentiable Penalties under High-dimensional Settings}

\begin{aug}
    \author{\fnms{Haolin} \snm{Zou,}\thanksref{m2}}
    \author{\fnms{Arnab} \snm{Auddy,}\thanksref{m1}}
    
    \author{\fnms{Kamiar} \snm{Rahnama Rad,}\thanksref{m3}}
    \and
    \author{\fnms{Arian} \snm{Maleki}\thanksref{m2}}
    
    \address{\thanksmark{m1} University of Pennsylvania\\
        \thanksmark{m2} Columbia University\\
        \thanksmark{m3} Baruch College, City University of New York
        }
\end{aug}

\begin{abstract}
Despite a large and significant body of recent work focused on estimating the out-of-sample risk of regularized models in the high dimensional regime, a theoretical understanding of this problem for non-differentiable penalties such as generalized LASSO and nuclear norm is missing. In this paper we resolve this challenge. We study this problem in the proportional high dimensional regime where both the sample size $n$ and number of features $p$ are large, and $n/p$ and the signal-to-noise ratio (per observation) remain finite. We provide finite sample upper bounds on the expected squared error of  leave-one-out cross-validation ($\LO$) in estimating the out-of-sample risk. 
The theoretical framework presented here provides a solid foundation for elucidating empirical findings that show the accuracy of $\LO$.
\end{abstract}

\end{frontmatter}

\section{Introduction}
\subsection{Main goal}
Estimating the out-of-sample risk (OO) to assess model performance and to set model complexity is a crucial task in statistical learning. A large body of recent work has studied the theoretical and empirical properties of cross-validation (or its variants) for estimating OO \citep{rad2020error,rad2018scalable,patil2021,patil2022,patil5failures,xu2021consistent,bellec2023out,beirami2017optimal,HaolinOnlineAppendix,stephenson2020approximate,stephenson2021can,luo2023iterative}. 

In spite of these recent advances, a complete quantitative description of the risk in using leave-one-out cross-validation (LO)  is missing when it comes to problems that have the following characteristics 1) non-differentiable penalties such as generalized LASSO and nuclear norm, and 2) finite signal-to-noise ratio when the dimension of the feature space grows proportionally with the number of observations. In this paper, we focus on the  aforementioned  characteristics for regularized generalized linear models (including logistic and Poisson regression). We use minor assumptions about the data generating process to provide finite sample upper bounds on the expected squared error of leave-one-out cross-validation in estimating the out-of-sample error. A consequence of our theory is that  leave-one-out cross-validation is a  consistent estimator of the out-of-sample risk. 

\subsection{Problem statement and related work}\label{sec:prblm-state}
Consider the dataset $\mathcal{D}=\{(y_i,\mathbf{x}_i)\}_{i=1}^n$, where $\mathbf{x}_i\in\mathbb{R}^p$ and $y_i\in \mathbb{R}$ denote the features and response of the $i^{\rm th}$ data point, respectively. Observations are independent and identically distributed (iid) draws from some unknown joint distribution $q(y_i|\mathbf{x}_i^\top\bbeta^*)p(\mathbf{x}_i)$, where $\bbeta^* \in \mathbb{R}^p$ represents the true parameter. 

The class of estimates, known as regularized empirical risk minimizers (R-ERM),  are based on the  following optimization:
\[
    \hbbeta:=\underset{\bbeta\in \bTheta}{\argmin}\sum_{i=1}^n\ell(y_i, \bx_i^\top\bbeta)+\lambda r(\bbeta),
    \numberthis\label{eq:opt:prob}
\]
where $\ell(y,z)$ is a non-negative convex loss function between $y$ and $z$, e.g., square loss $\ell(y,z)=\frac12 (y-z)^2$, $r(\bbeta)$ is a non-negative convex regularizer, and $\bTheta$ is a convex set to which $\bbeta^*$ belongs. By selecting particular forms for $\ell$ and $r$ we can cover a wide range of popular estimators within the high-dimensional setting. For instance, one may use logistic or Poisson for the loss function, and fused LASSO or nuclear norm for the regularizer. 

One widely used criterion for model selection and evaluating the accuracy of $\hat{\bbeta}$ is the out-of-sample prediction error ($\OO$), defined as
$$
\OO :=\EE[\phi(y_0,\bx_0^\top\hbbeta)|\cD],
\label{eq:OO}
$$
where $\phi(y,z)$ is a function measuring the difference between $y$ and $z$ (often, but not necessarily the same as $\ell(y,z)$), and $(y_0,\bx_0)$ is a sample from the same joint distribution $q(y|\bx^\top\bbeta^*)p(\bx)$, independent of the training set $\cD=\{(y_i,\bx_i)\}_{i=1}^n$. 

Several methods are proposed to estimate $\OO$ such as
$K$-fold cross-validation, Generalized Cross-Validation, and Bootstrap, among others. In this paper, we employ the term ``risk estimate" as a generic term for referring to these techniques. Despite the abundance of theoretical and empirical results in the literature, studies specifically examining the accuracy of risk estimates in high-dimensional settings, where the number of features is either larger than or comparable to the sample size $n$, are noticeably scarce. As a result many fundamental questions regarding the performance of risk estimates have remained open. 

The main objective of this paper is to prove the accuracy of the leave one out risk estimate $\LO$ for a large class of R-ERMs under the high-dimensional setting. 
For $1\leq i\leq n$, define
\[
    \hbbeta_{/i}:=\underset{\bbeta\in \bTheta}{\argmin}\sum_{j\neq i} \ell(y_j, \bx_j^\top\bbeta)+\lambda r(\bbeta)\numberthis\label{eq:beta_wo_i}.
\]
$\LO$ is defined as:
\[
    \LO:= \frac1n \sum_{i=1}^n\phi(y_i,\bx_i^\top\hbbeta_{/i}).
\]
The success of $\LO$ in practical scenarios has encouraged numerous researchers to investigate its accuracy and develop several computationally-efficient approximations for it \citep{rad2020error,rad2018scalable,patil2021,patil2022,patil5failures,xu2021consistent,bellec2023out,beirami2017optimal,HaolinOnlineAppendix,stephenson2020approximate,stephenson2021can,luo2023iterative, austern2020asymptotics}. Notably, key references that are related to our current work are \citep{rad2020error, burman88}. In their study, \citep{burman88} studied the accuracy of $\LO$ in the low-dimensional setting, $p$ fixed, while $n \rightarrow \infty$ and 
established the consistency of $\LO$. In other words, they proved that $\LO
\stackrel{p}{\to}\OO$. In the more recent paper \citep{rad2020error}, the authors studied the accuracy of the $\LO$ under the high-dimensional setting, similar to the setting we consider in this paper. However, it is worth noting that the conclusions drawn in \citep{rad2020error} are contingent upon the following two assumptions that restrict the broader applicability of their outcomes:
\begin{enumerate}
    \item Restriction 1: The regularizer is twice differentiable. 
    \item Restriction 2: The set $\bTheta$ is the same as $\mathbb{R}^p$. 
\end{enumerate}

The first assumption excludes a majority of crucial regularizers commonly employed in practical scenarios, including the $\ell_1$ norm, nuclear norm, group-LASSO, etc. Additionally, the second assumption eliminates scenarios where parameters adhere to a specific set. Instances of such situations arise when a practitioner seeks to estimate a positive definite matrix or when the elements of $\bbeta^*$ are required to be positive, increasing, or both.

Our paper aims to eliminate Restriction 1 and Restriction 2, which constrained the applicability of the results in \citep{rad2020error} to numerous real-world problems, and still obtain a sharp upper bound on the error $\LO-\OO$.


The main technical novelties that have enabled us to achieve these goals are clarified in Section \ref{sec:main}. 

\subsection{Notations}\label{subsec:notations}

In this section, we summarize the symbols we use throughout the article. Vectors are denoted with boldfaced lowercase letter, such as $\bx$. Matrices are represented by boldfaced capital letters, such as $\bX$. For a matrix $\bX$, $\sigma_{\min}(\bX)$, $\|\bX\|$, $\tr(\bX)$ denote the minimum singular value, the spectral norm (equal to the maximum singular value $\sigma_{\max}(\bX)$), and the trace of the matrix $\bX$ respectively. 

We denote $\cD$ as the full dataset $\{(y_i,\bx_i)\}_{i=1}^n$, and $\cD_{/i}$ as the dataset excluding the $i^{\rm th}$ observation, i.e. $\cD_{/i} = \{ (y_j,\bx_j): j\neq i \}$. We also use $[n]:=\{1,2,...,n\}$ for any positive integer $n$.

For brevity we denote $\ell_j(\bbeta):=\ell( y_j, \bx_j^\top\bbeta )$, and similarly $\phi_j(\bbeta):=\phi(y_j, \bx_j^\top\bbeta)$. The following definitions are also used:
\begin{align*}
      \dot{\ell}_i(\bbeta) &:= \frac{\partial \ell(y_i ,z) }{\partial z}\bigg|_{z=\bx_i^\top\bbeta}, \dphi_i(\bbeta):= \frac{\partial \phi(y_i ,z)}{\partial z} \bigg|_{z=\bx_i^\top\bbeta}.
\end{align*}

Polynomials of $\log(n)$ are denoted by $\polylog(n)$. When it is specifically defined, a subscript is added. For $x,y\in\RR$, we write $x\wedge y$ and $x\vee y$ to denote $\min\{x,y\}$ and $\max\{x,y\}$ respectively.

The usage of $O(1)$ and $o(1)$ are conventional, and $a_n=\Theta(1)$ iff both $a_n$ and $a_n^{-1}$ are $O(1)$. Similarly, for a sequence of random variables, $X_n=O_p(1)$ means $X_n$ being stochastically bounded, i.e. $\exists C>0$ s.t. $\PP(|X_n|>C)\to 0$. Similarly $X_n=o_p(1)$ means $X_n$ converge to zero in probability, and $X_n=\Theta_p(1)$ iff 
both $X_n$ and $X_n^{-1}$ are $O_p(1)$.

For a closed subset $\bTheta$ of $\RR^p$, let $\cC^k(\bTheta)$ denote the collection of all functions on $\bTheta$ with continuous $k^{{\rm th}}$ Fr\'echet derivative on the interior of $\bTheta$. In particular, $\cC^0(\bTheta)$ consists of all continuous functions on $\bTheta$.

\subsection{Organization of our paper}
The rest of the paper is organized as follows. The assumptions are listed in Section \ref{subsec:assumptions}, followed by discussions on them in Section \ref{subsec:discussion_assumptions}. The main theorem is stated in Section \ref{ssec:mainthm}. In Section \ref{sec:proofs} we introduce the main challenges and proof techniques (Section \ref{subsec:challenge}) and a proof sketch of the main theorem (Section \ref{subsec:proof_main}). The concluding remarks are in Section \ref{subsec:conclusion}.  Detailed proofs of the lemmas are postponed to the Appendix.

\section{Main result}\label{sec:main}
\subsection{Assumptions}\label{subsec:assumptions}
As we discussed in Section \ref{sec:prblm-state} our goal is to prove the accuracy of LO under the high-dimensional settings. Before we state our main theorem we explain some of the assumptions we have made and discuss their validity. 

\begin{enumerate}
    \item[A1] $\bTheta \subset \RR^p$ is a closed convex set.

    \item[A2] Both $n$ and $p$ are large, while $n/p\equiv\gamma_0\in (0,\infty)$.

    \item[A3] $\bX = (\bx_1,\cdots, \bx_n)^\top$ where $\bx_i \in \mathbb{R}^p$ are i.i.d. $N(0,\bSigma)$ samples. Moreover, there exist constants $0<c_{X}\leq C_{X}$ such that $p^{-1} c_X \leq \sigma_{\min}(\bSigma)\leq \sigma_{\max}(\bSigma)\leq p^{-1}C_X$.

    \item[A4] $r(\bbeta)=(1-\eta)r_0(\bbeta) + \eta\bbeta^\top\bbeta$ where $\eta \in (0,1)$ and $r_0$ is non-negative, convex and Lipschitz continuous. 

    \item[A5] Assume that $\ell(y,z)$ and $\phi(y,z)$ are non-negative, convex and continuously differentiable with respect to $z$ and that $\ell$, $\dl$, $\phi$ and $\dphi$ (defined in Section \ref{subsec:notations}) grow polynomially in $y, z$, i.e., there exist constants $s$ and $C>0$ such that 
    \begin{align*}
        &\max\{|\ell(y,z)|, |\dl(y,z)|,|\phi(y,z)|, |\dphi(y,z)|
        \}\\
        \leq& C(1+|y|^s+ |z|^s)
    \end{align*}
    for all $(y,z)$. Furthermore, assume that all moments of $|y_i|$ are bounded, i.e. $\forall m>0$ , $\exists$ constant $C_Y(m)$ depending only on $m$, such that $\forall n\geq 1, \forall i \in [n]$:
    \[
        \EE |y_i|^m \leq C_Y(m).
    \]
\end{enumerate}

\subsection{Discussion of Assumptions}\label{subsec:discussion_assumptions}

The goal of this section is to provide further intuition on the assumptions we introduced in the previous section. 

Assumption A1 is standard in convex optimization. Assumption A2 has become one of the standard frameworks for studying high-dimensional problems \citep{obuchi2019cross, MMB15, xu2021consistent, miolane2021distribution, donoho2009message, donoho2011noise, bu2019, patil2021, patil2022, jalali2016, bellec2023out, celentano2023lasso, li2022non, liang2022precise, fan2022approximate, dudeja2022universality, dudeja2023universality}, since it has been able to prove some of the peculiar features estimators exhibit in high-dimensional settings, such as phase transitions \citep{WeMaZh2018}.   

The  Gaussianity of $\bx_i$ in Assumption A3 is prevalent in theoretical papers dealing with high-dimensional problems. While it is straightforward to relax this assumption to sub-Gaussianity, we focus on Gaussian distribution to keep the notations clear and simple. The scaling we have adopted here, i.e.,
\[
p^{-1} c_X \leq \sigma_{\min}(\bSigma)\leq \sigma_{\max}(\bSigma)\leq p^{-1}C_X,
\]
 is based on the following rationale. First note that 
\[
\frac{c_X}{p} \|\bbeta^*\|_2^2 \leq \mathbb{E} (\bx_i^\top \bbeta^*)^2 \leq \frac{C_X}{p} \|\bbeta^*\|_2^2. 
\]
When $n, p \rightarrow \infty$ with $n/p \equiv \gamma_0$ and when the elements of $\bbeta^*$ are $O(1)$,  we have $\|\bbeta^*\|=O(\sqrt{p})$, and hence $\mathbb{E} (\bx_i^\top \bbeta^*)^2 = O(1)$. Therefore, under the settings of the paper we can see that the signal-to-noise ratio (SNR) of each data point, defined as $\frac{\var(\bx_i^\top\bbeta^*)}{\var(y_i|\bx_i^\top\bbeta^*)}$, remains bounded. The reader may check Appendix \ref{subsec:bd_snr} for details. Intuitively speaking, if SNR per datapoint is very large, the estimation problem will be easy and the problem of risk estimation is not of particular interest; maximum likelihood estimators perform very well. On the other hand, if SNR per data point is very small, the accurate estimation of $\hbbeta$ will not be possible unless we impose conditions on $\bbeta^*$, such as stringent sparsity assumption. However, since we want our results to be generic in terms of $\bbeta^*$, we do not want to impose any stringent constraint on $\bbeta^*$.

Another way we can justify the scaling in Assumption A3 is that under this scaling, neither the loss $\sum_j \ell(y_j,\bx_j^\top\bbeta^*)$ nor the regularizer $\lambda r(\bbeta^*)$ dominate each other when $n,p\to \infty$. In other words, for a large class of losses $\sum_j \ell(y_j,\bx_j^\top\bbeta^*)= O_p(n)$ and for a large class of regularizers $\lambda r(\bbeta^*)= O_p(p)$. Given that $n/p=\gamma_0$, the two terms have the same order. Effectively, this makes the optimal choice of $\lambda$ (that gives the minimum out-of-sample error), $O_p(1)$. For more precise arguments behind this, we refer the reader to \citep{MMB15, WaWeMa20, wang2022does}.

Assumptions A4 guarantees the loss function to be $2\lambda\eta$-strongly convex.  Researchers have noticed that in wide range of applications, adding $\eta \bbeta^{\top} \bbeta$ in addition to the non-differentiable regularizer often improves the prediction performance \citep{hastie2017extended, mazumder2023subset, wang2022does, guo2023signal}.

Assumption A4 also posits the Lipschitz continuity of $r_0$. Nearly all popular non-differentiable regularizers satisfy this condition. Appendix \ref{subsec:assumption_a4} presents several examples including LASSO, generalized LASSO and Schatten norms with the nuclear norm as a special case.

Assumption A5 also holds for a wide range of models that are used in practice. For instance Appendix \ref{subsec:assumption_a5} considers standard data generation mechanisms that are used in linear regression, logistic regression and Poisson regression, and show that for all those models Assumption A5 holds.

\subsection{Main theorem}\label{ssec:mainthm}
The following theorem and its discussion are the main contributions of this paper:

\begin{theorem}\label{thm:main}
    Under Assumption A1-A5, there exists a constant $C_v$ depending only on $\{\gamma_0,\lambda,\eta, s, C_X, C_Y(\cdot)\}$ such that
        \begin{align*}
            \EE ( \LO - \OO )^2
             &\leq \frac{C_v}{n}.
        \end{align*}
  \end{theorem}

Before we present the sketch of the proof, we would like to discuss this theorem and provide some intuition.\\

\begin{remark}
A simple conclusion of this theorem is that under the asymptotic setting in which $n,p \rightarrow \infty$ and the ratio $n/p$ remains fixed, and $p^{-1} c_X \leq \sigma_{\min}(\bSigma)\leq \sigma_{\max}(\bSigma)\leq p^{-1}C_X$, LO offers a consistent estimate of the out-of-sample prediction error in the sense that $\LO \rightarrow \OO$ in probability. However, note that Theorem~\ref{thm:main} offers more than the consistency, and it captures the convergence rate as well.\\
\end{remark}

\begin{remark}
The rate $\frac{1}{n}$ obtained in Theorem \ref{thm:main} is expected to be sharp. Note that the leave-one-out cross-validation takes an average of $n$ estimates of the OO, namely $\phi(y_i; \bx_i^{\top} \hbbeta_{/ i})$. The variance of each of these estimates is $O(1)$. Hence, if all these estimates were independent, the variance of LO would be still proportional to $O(1/n)$, which is the same as the bound we have obtained. \\
\end{remark}

\begin{remark}
The rate $O(1/n)$ has also been seen in the previous work on the analysis of LO under the low-dimensional asymptotic, where $p$ is fixed, while $n \rightarrow  \infty$. For instance, \citep{burman88} provided a rate estimate of the variance of $\LO$ (Theorem 6.2(c) in \citep{burman88}, notations modified) and showed:
    \begin{align*}
        &\var(\LO - \OO) = O(1/n).
    \end{align*}
    Note that $\var(\LO- \OO)\leq \EE(\LO- \OO)^2$. 
    Our result show that the rate is still $O(1/n)$ in high-dimensional settings. More recently, the paper \citep{austern2020asymptotics} has characterized the limiting distribution of the $k$-fold cross-validation under the low-dimensional setting. In order to obtain a non-degenerate limiting distribution for LO, \citep{austern2020asymptotics} has to scale $\LO- \OO$ with $\sqrt{n}$. This scaling translates to the same scaling as the one in \citep{burman88}. 
\end{remark}


\section{Proofs}\label{sec:proofs}

\subsection{Main challenges and novel techniques} \label{subsec:challenge}

As mentioned earlier, \citep{rad2020error} established a result akin to Theorem \ref{thm:main} in the case where the regularizer is twice differentiable, and there is no constraint, i.e. $\bTheta = \mathbb{R}^p$. 

In our proof, there are two new challenges: the non-smoothness of the regularizer $r$, and the existence of the convex constraint $\bTheta$. We introduce two novel elements that allowed us to significantly broaden the scope of Theorem \ref{thm:main} well beyond what was offered by \citep{rad2020error}: smoothing and projection.

\begin{enumerate}
\item Smoothing: We start with approximating the non-smooth regularizer, $r_0(\bbeta)$ with a smooth function: 
\[r_0^\alpha(\bbeta) = \int_{\bTheta}r_0(\bbeta-\bz)\alpha\phi(\alpha\bz)d\bz \numberthis\label{eq:gaussian_ralpha}\]

where $\phi(\bz) = (2\pi)^{-p/2}e^{-\frac12\bz^\top\bz}$ is the density of a standard Gaussian vector.\footnote{Note that when $\bbeta$ is a $p_1$ by $p_2$ matrix, we concatenate its rows to form a vector of dimension $p=p_1\times p_2$.} The following lemma shows one of the properties of this approximation that will be used throughout the paper. 

\begin{lemma}\label{lem:gaussian_smooth}
    Suppose $\bTheta$ is closed and  $r\in \cC^0(\bTheta)$. 
    \begin{enumerate}
        \item Suppose there exists a positive integer $K$ such that $r(\bz) \norm{\bz}_2^k e^{-\frac12 \norm{\bbeta-\bz}_2^2}$ is integrable. Then $r^\alpha$ defined in \eqref{eq:gaussian_ralpha} is in $\cC^k(\bTheta_0)$ where $\bTheta_0$ is the interior of $\bTheta$.
        \item Suppose there exist constants $L>0$ and $k\in (0,1]$ such that for all $\bx,\by\in\bTheta$,
        \[
            |r(\bx) - r(\by)|\leq L \norm{\bx-\by}_2^k.
        \]
        Then as $\alpha\to\infty$, we have
        \[
            \norm{r^\alpha - r}_\infty \to 0. 
        \]
    \end{enumerate}        
\end{lemma}
The proof of this lemma is presented in Appendix \ref{subsec: approx_r}. Note that all assumptions in Lemma \ref{lem:gaussian_smooth} are satisfied when $r$ is Lipschitz continuous as in Assumption A4.
This lemma implies that $r_0^{\alpha} (\bbeta)$ is an accurate approximation of $r_0(\bbeta)$ for all values of $\bbeta$. Define
\begin{align*}
    \hbbeta^\alpha &:= \underset{\bbeta\in\bTheta}{\argmin} \sum_{j=1}^n \ell_j(\bbeta) + \lambda (1-\eta) r_0^{\alpha}(\bbeta) + \lambda \eta \bbeta^\top\bbeta\\
    \hbbeta^\alpha_{/i} &:= \underset{\bbeta\in\bTheta}{\argmin} \sum_{j \neq i} \ell_j(\bbeta) + \lambda (1-\eta) r_0^{\alpha}(\bbeta) + \lambda \eta \bbeta^\top\bbeta.
\end{align*}
 Our final goal is to show that by analyzing the $\LO$ for the surrogate estimates $\hbbeta_{/ i}^\alpha$ and $\hbbeta^\alpha$ we can also analyze the $\LO$ for the original estimates $\hbbeta_{/ i}$ and $\hbbeta$. This will be clarified in Section \ref{subsec:proof_main}.

\item Projection operator: The second challenge we have to address in proving Theorem \ref{thm:main} is the existence of the constraint set $\bTheta$. One can write the optimization problem $\hbbeta=\underset{\bbeta\in \bTheta}{\argmin}\sum_{i=1}^n\ell_i(\bbeta)+\lambda r(\bbeta)$ as
\begin{eqnarray}
\hbbeta=\underset{\bbeta}{\argmin}\sum_{i=1}^n\ell_i(\bbeta)+\lambda r(\bbeta) + {\bI}_{\bTheta} (\bbeta), \nonumber
\end{eqnarray}
where ${\bI}_{\bTheta} (\bbeta)$ is the convex indicator function of the set $\bTheta$, and treat the constraint as another non-differentiable regularizer and use smoothing again. However,  indicator function ${\bI}_{\bTheta} (\bbeta)$ is not a Lipschitz function and cannot be uniformly approximated by smooth functions. Hence the smoothing argument discussed before will not work. Hence, we pick a different approach and represent $\hbbeta^\alpha$ in a different way. The following lemma provides this alternative representation:
\begin{lemma}\label{lem:fix_eq_prox}
    Suppose 
    \begin{itemize}
        \item $R:\RR^p \to \RR\cup \{-\infty,\infty\}$ is proper convex with $dom(R)$ being closed.
        \item $L:\RR^p\to \RR$ is differentiable on the relative interior of $dom(R)$ and proper convex on $dom(R)$.
    \end{itemize} 
    Define the proximal operator of the function $R$ as
    \[
    \prox_R(\mathbf{u}) := \argmin_{\bx}\left\{ R(\bx)+ \frac{1}{2} \|\mathbf{u}-\bx\|_2^2\right\}. 
    \]
    Then a solution of the following equation:
    \[
        \bx = \prox_R(\bx-\nabla L(\bx))
        \label{eq:fix_eq_prox}\numberthis
    \]    
    is also a minimizer of the problem
    \[
        \min_{\bx} \left\{L(\bx) + R(\bx)\right\}
        \label{eq:min_L_R}\numberthis
    \]
    and vice versa.
\end{lemma}
\begin{proof}
    See Propositon 3.1 (iii)(b) in \citep{cw05proximity}. 
\end{proof}
Using this lemma we find a new representation for $\hbbeta$. Consider the problem with a convex constraint: 
    \[
        \hbbeta = \underset{\bbeta\in \bTheta}{\argmin} \sum_i \ell_i(\bbeta) + \lambda r(\bbeta).
    \]
    where $\ell$ and $r$ are smooth, $\bTheta$ is a closed convex set. Using Lemma \ref{lem:fix_eq_prox} with $L(\bbeta)=\sum_i \ell_i(\bbeta) + \lambda r(\bbeta)$ and $R(\bbeta) = \bI_{\bTheta}(\bbeta)$ where $\bI_{\bTheta}(\bbeta)$ is the convex indicator function\footnote{Convex indicator function of set $\bTheta$ is defined as: $\bI_{\bTheta}(\bbeta)= 0$ if $\bbeta\in \bTheta$ and $\bI_{\bTheta}(\bbeta)=+\infty$ otherwise}, we have that $\hbbeta$ should satisfy
    \begin{align*}
        \hbbeta &= \prox_{\bI_{\bTheta}}(\hbbeta - \sum_i \dl_i(\hbbeta)\bx_i - \lambda \nabla r(\hbbeta))\\
        &= \bPi_{\bTheta} (\hbbeta - \sum_i \dl_i(\hbbeta)\bx_i - \lambda \nabla r(\hbbeta))\label{eq:fix_eq_proj},\numberthis
    \end{align*}
    where $\bPi_{\bTheta}$ is the metric projection operator
    \[
        \bPi_{\bTheta}(\bx):=\underset{\bz\in \bTheta }{\argmin}\|\bx-\bz\|_2.
    \]
Note that to obtain the last equality in \eqref{eq:fix_eq_proj}, we have used the fact that the proximal operator of $\bI_{\bTheta}$ is the same as the metric projection  onto set $\bTheta$.  

Using \eqref{eq:fix_eq_proj}, instead of representing $\hbbeta^\alpha$ as the minimizer of $\sum_{i=1}^n\ell_i(\bbeta)+\lambda r^{\alpha}(\bbeta) + {\bI}_{\bTheta} (\bbeta)$, we represent it as the solution of the following fix point equation: 
\begin{align*}
		\hbbeta^\alpha = \bPi_{\bTheta}(\hbbeta^\alpha-\nabla h^{\alpha}(\hbbeta^\alpha))
\end{align*}
where $\bPi_{\bTheta} (\cdot)$ denotes the metric projection onto $\bTheta$. Similarly,
$\hbbeta_{/i}^\alpha$ satisfies:
\begin{align*}
		\hbbeta_{/i}^\alpha = \bPi_{\bTheta}(\hbbeta_{/i}^\alpha-\nabla h_{/i}^{\alpha}(\hbbeta_{/i}^\alpha)),
	\end{align*}
The following lemma establishes some of the important properties of the projection operators that are particularly useful in our paper:

\begin{lemma}\label{lem:proj_property}
    Suppose $\bTheta$ is a closed convex set in $\RR^p$, and $\bx, \by\in \RR^p$. Let $\bPi_{\bTheta}$ be the metric projection onto $\bTheta$. Then there exists a matrix-valued function $\bJ(t)$ with 
    \[
        0\leq \lambda_{\min}(\bJ(t))\leq \lambda_{\max}(\bJ(t))\leq 1,\quad \forall t\in [0,1]
    \]
    such that
    \[
        \bPi_{\bTheta}(\by) - \bPi_{\bTheta}(\bx) = \int_0^1 \bJ(t)dt(\by-\bx).
    \]
\end{lemma}
For completeness, we include the proof of this well known result in Section \ref{subsec:proof_proj_property}.
\end{enumerate}

In the next section, we explain how using the surrogate estimates $\hbbeta_{/ i}^\alpha$ and $\hbbeta^\alpha$ and representing them as the solution of, e.g. 
\begin{align*}
		\hbbeta_{/i}^\alpha = \bPi_{\bTheta}(\hbbeta_{/i}^\alpha-\nabla h_{/i}^{\alpha}(\hbbeta_{/i}^\alpha)),
	\end{align*}
enable us to obtain an upper bound on the difference $\LO-\OO$. 

\subsection{Proof summary of Theorem \ref{thm:main}}\label{subsec:proof_main}
In this subsection, we present a brief sketch of the proof. Similar to the proof of Theorem 1 in \citep{rad2020error} we use the following two definitions:
\begin{align*}
    V_1 &=  \LO  - \frac{1}{n} \sum_{i\in[n]}\EE[ \phi_i(\hbbeta_{/i})   |\cD_{/i}],      \\
    V_2 &=  \frac{1}{n} \sum_{i\in[n]}\EE[ \phi_i(\hbbeta_{/i})   |\cD_{/i}]  -   \OO,
    \end{align*}
    and use the following upper bound on $\EE \left[ \LO - \OO \right]^2$:
    \begin{align*}
    \EE \left[ \LO - \OO \right]^2
    = \EE \left ( V_1 + V_2  \right )^2 
    \leq 2\EE V_1^2 + 2\EE V_2^2.
    \label{eq:v1+v2}\numberthis
\end{align*}
Hence, the remaining steps are to obtain upper bounds for $\EE V_1^2$ and $\EE V_2^2$. To see how the ideas we introduced in Section \ref{subsec:challenge} enable us to obtain the required upper bound, we provide the details of proving an upper bound for $\EE V_2^2$ (which is shorter) here, and postpone the problem of finding an upper bound for $\EE V_1^2$ to Appendix \ref{sec:pf-lemv1}.

To bound $\EE V_2^2$, first notice that for all $j\neq i$ we have
\begin{align*}
    \EE[ \phi_j(\hbbeta_{/i}) | \cD_{/i}] &= \EE[ \phi_0(\hbbeta_{/i}) |  \cD_{/i}] = \EE[ \phi_0(\hbbeta_{/i})  | \cD].
\end{align*}
By the mean-value theorem, for each $i$ there exists a random variable $\bxi_i=t_i\hbbeta_{/i} + (1-t_i)\hbbeta$ with $t_i \in [0, 1]$ such that
\[
    \phi_0(\hbbeta_{/i}) - \phi_0(\hbbeta) = \dphi_0(\bxi_i) \bx_0^\top (\hbbeta_{/i} - \hbbeta).
\]
Then we have
\begin{align}\label{eq:v2:part1}
  \EE (V_2^2) &= \EE \left( \frac{1}{n} \sum_{i=1}^n   \EE[ \phi_i(\hbbeta_{/i})  | \cD_{/i}]   - \EE [\phi_0( \hbbeta) |  \cD]  \right)^2 \nonumber \\
    &= \EE \left( \frac{1}{n} \sum_{i=1}^n   \EE[ \phi_0( \hbbeta_{/i})| \cD]   - \EE [\phi_0( \hbbeta)  | \cD]  \right)^2 \nonumber
    \\
    &= \frac{1}{n^2}\EE \left( \sum_{i=1}^n  \EE[\dphi_0(\bxi_i) \bx_0^\top (\hbbeta_{/i} - \hbbeta)|   \cD]   \right)^2 
    \nonumber \\
    &= \frac{1}{n^2}\sum_{i=1}^n\sum_{j=1}^n 
    \EE\Big(\EE[\dphi_0(\bxi_i) \bx_0^\top (\hbbeta_{/i} - \hbbeta)|\cD]
    \nonumber \\
    &\hspace{2cm} \cdot
    \EE[\dphi_0(\bxi_j) \bx_0^\top (\hbbeta_{/j} - \hbbeta)|\cD]\Big)\nonumber\\
    &\leq \frac{1}{n^2}\sum_{i=1}^n\sum_{j=1}^n 
    \sqrt{\EE\Big(\EE[\dphi_0(\bxi_i) \bx_0^\top (\hbbeta_{/i} - \hbbeta)|\cD]\Big)^2}
    \nonumber \\
    &\hspace{2cm} \cdot
    \sqrt{\EE\Big(\EE[\dphi_0(\bxi_j) \bx_0^\top (\hbbeta_{/j} - \hbbeta)|\cD]\Big)^2}
    \nonumber\\
    &=\EE\Big(\EE[\dphi_0(\bxi_1) \bx_0^\top (\hbbeta_{/1} - \hbbeta)|\cD]\Big)^2,
\end{align}
where we use Cauchy Schwarz inequality in the penultimate step. Next we have
\begin{align*}
    &\abs{\EE[\dphi_0(\bxi_i) \bx_0^\top (\hbbeta_{/i} - \hbbeta)|\cD]}\\
    \leq& \sqrt{\EE[ \dphi_0^2(\bxi_i) |\cD] \EE[ (\hbbeta_{/i} - \hbbeta)^\top\bx_0\bx_0^\top(\hbbeta_{/i} - \hbbeta)|\cD] }\\
    \leq & \sqrt{\EE[ \dphi_0^2(\bxi_i) |\cD]}\sqrt{\frac{C_X}{p}\|\hbbeta_{/i} - \hbbeta\|^2}\label{eq:v2:part2}\numberthis
\end{align*}
where the last inequality uses the independence between $\bx_0$ and $\cD$ and also the fact that $\sigma_{\max}(\bSigma)\leq \frac{C_X}{p}$ by Assumption A3.

Hence, in order to bound $\mathbb{E} (V_2^2)$, we have to obtain bounds on $\|\hbbeta - \hbbeta_{/i}\|$ and also $\EE[\dphi_0^2(\bxi_i)|\cD]$. That is what the next two lemmas aim to do. The first lemma connects $\hbbeta_{/i}$ and $\hbbeta$:

\begin{lemma}\label{lem:beta_lo_error}
Under assumptions A1-A4, we have for all $\alpha>0$ that
    \begin{align*}
        \norm{\hbbeta - \hbbeta_{/i}}\leq \frac{|\dl_i(\hbbeta_{/i})|\norm{
        \bx_i}}{2\lambda\eta\wedge 1}.
    \end{align*}
\end{lemma}

The proof of this lemma is presented in Section \ref{subsec:proof_beta_lo_error} and uses the ideas that we mentioned in Section \ref{sec:main}, i.e. smoothing and projection operator.

The next lemma bounds the moments of $\phi_0$ and $\dphi_0$:

\begin{lemma}\label{lem:mean_dphi0}
    Suppose assumptions A1-A5 hold. Then there exists a constant $C_{\phi}$ depending only on $s$ (from Assumption A5) and $C_X$ (from Assumption A3) such that $\forall \bbeta\in \RR^p$:
    \begin{align*}
        \sqrt{\EE[\phi_0^2(\bbeta)]}&\leq C_{\phi} + C_{\phi}\left( \frac1p \|\bbeta\|^2 \right)^{s/2},
        \\
        \sqrt{\EE[\dphi_0^2(\bbeta)]}&\leq C_{\phi} +C_{\phi}\left( \frac1p \|\bbeta\|^2 \right)^{s/2},
        \\
        \sqrt{\EE[\dphi_0^4(\bbeta)]}&\leq C_{\phi}^2 +C_{\phi}^2\left( \frac1p \|\bbeta\|^2 \right)^{s}.
    \end{align*}
\end{lemma}
The proof of this lemma can be found in Section \ref{subsec:proof_mean_dphi0}.


Inserting Lemma~\ref{lem:beta_lo_error} and Lemma~\ref{lem:mean_dphi0} back into \eqref{eq:v2:part2} we have
\begin{align*}
    &\abs{\EE[\dphi_0(\bxi_1) \bx_0^\top (\hbbeta_{/1} - \hbbeta)|\cD]}
    \\
    \leq& \left(C_{\phi}+C_{\phi}\left(\frac1p\|\bxi_1\|^2\right)^{\frac{s}{2}}\right)\sqrt{\frac{C_X}{p}}\frac{|\dl_1(\hbbeta_{/1})|\|\bx_1\|}{2\lambda\eta\wedge 1}. 
\end{align*}
Hence, if we use \eqref{eq:v2:part1}, then we will obtain 

\begin{align*}
    &\EE V_2^2
    \\
    \leq& \EE\left( 
    C_{\phi}^2\left(1+\frac{\|\bxi_1\|^s}{p^{s/2}}\right)^2
    \frac{C_X}{p(2\lambda\eta\wedge 1)^2}\dl_1^2(\hbbeta_{/1})\|\bx_1\|^2\right)
    \\
    \overset{(a)}{\leq}& \frac{C_{\phi}^2C_X}{p(2\lambda\eta\wedge 1)^2} \sqrt{\EE\left(
    1+\frac{1}{p^{s/2}}\|\bxi_1\|^s
    \right)^4}\times\\
    &~\hspace{2cm}\times 
    \sqrt{\EE\left(\dl_1^4(\hbbeta_{/1})\|\bx_1\|^4\right)}
    \\
    \overset{(b)}{\leq}& \frac1n \frac{C_{\phi}^2 C_X\gamma_0}{(2\lambda\eta\wedge 1)^2}
    \sqrt{8\left(1+\EE\left(\frac1p\|\bxi_1\|^2\right)^{2s}\right)}\\
    &\times \left(\EE\dl_1^8(\hbbeta_{/1})\right)^{\frac14} \cdot \left(\EE\|\bx_1\|^8\right)^{\frac14}\label{eq:v2:part3}.\numberthis
\end{align*}
where steps (a) and (b) both used Cauchy Schwarz Inequality.
The following lemma bounds $\EE \left(\frac1p\|\hbbeta\|^2\right)^{t}$ and $\EE\dl_1^8(\hbbeta_{/1})$:

\begin{lemma}\label{lem:assumption_a5}
    Under assumptions A1-A5, there exist constants $C_{\beta}(t)$ depending on $\{\gamma_0,\lambda, \eta, C_Y(\cdot), s, t\}$ and $C_{\ell}$ depending on $\{\gamma_0,\lambda, \eta, C_Y(\cdot), s, C_X\}$ such that
    \begin{enumerate}
        \item[(a)] For $t\geq 1$,
        \begin{align*}
            \EE[p^{-1}\|\hbbeta\|^2]^{t}&\leq C_{\beta}(t)\\
            \EE[p^{-1}\|\hbbeta_{/1}\|^2]^{t}&\leq C_{\beta}(t),
        \end{align*}
        \item[(b)] $\EE\dl_1^8(\hbbeta_{/1})\leq C_{\ell}$.
    \end{enumerate}
\end{lemma}

The proof of this lemma can be found in Section~\ref{subsec:assn_a5}. Observe that $\bxi_1$ is a convex combination of $\hbbeta$ and $\hbbeta_{/i}$, then using part (a) of Lemma \ref{lem:assumption_a5} we have that the same bound applies to $p^{-1}\|\bxi_1\|^2$. Next, by Lemma \ref{lem:sum_xi_conc} we have $\EE\|\bx_1\|^8\leq 24C_X^4$. Inserting the above bounds  into \eqref{eq:v2:part3} we have 
\begin{align*}
    \EE V_2^2 
    &\leq \frac1n \frac{C_{\phi}^2 C_X\gamma_0}{(2\lambda\eta\wedge 1)^2}
    \sqrt{8\left(1+C_{\beta}(2s)\right)}
    \cdot C_{\ell}^{\frac14} \cdot (24C_X^4)^{\frac14}
    \\
    &\leq \frac1n \frac{7C_{\phi}^2 C_X^2C_{\ell}^{1/4}\gamma_0}{(2\lambda\eta\wedge 1)^2}
    \sqrt{1+C_{\beta}(2s)}
    \\
    &:= \frac{C_{v2}}{n}.
    \label{eq:V2bd}\numberthis
\end{align*}
The second line uses the fact that $\sqrt{8}(24)^{1/4}\leq 7$.

Obtaining an upper bound for $\mathbb{E} (V_1^2)$ uses similar techniques, although it involves more cumbersome calculations. In fact we can prove that

\begin{lemma}\label{lem:v1}
Under assumptions A1-A5, for large enough $n,p$ we have:
\begin{align*}
\EE V_1^2
\leq \dfrac{C_{v1}}{n}
\end{align*}
for $C_{v1}>0$ depending only on $\{\gamma_0,\lambda, \eta, C_Y(\cdot), s, C_X\}$.
\end{lemma}
The proof of Lemma~\ref{lem:v1} is given in Section~\ref{sec:pf-lemv1}.

Combining Lemma \ref{lem:v1} and \eqref{eq:V2bd} with \eqref{eq:v1+v2} completes the proof of Theorem~\ref{thm:main}:
\[
    \EE[\LO-\OO]^2\leq 2\EE V_1^2 + 2\EE V_2^2 \leq \frac{2(C_{v1}+C_{v2})}{n}:=\frac{C_v}{n}.
\]

\section{Conclusion}\label{subsec:conclusion}
In this paper, our focus was on the class of regularized empirical risk minimization (R-ERM) techniques that incorporate non-differentiable regularizers. We studied the accuracy of leave-one-out cross-validation techniques within a high-dimensional setting, where both the number of observations $n$ and the number of features $p$ are large while the ratio $n/p$ is bounded. We derived a finite-sample upper bound for the difference between the out-of-sample prediction error and its leave-one-out cross-validation estimate. Our upper bound shows that if $\OO$ and $\LO$ represent the out-of-sample prediction error and its leave-one-out estimate, then $\EE (\LO-\OO)^{2} = O(\frac{1}{n})$.  

\section*{Acknowledgments}
Arian Maleki would like to thank NSF (National Science Foundation) for their generous support through grant number DMS-2210506. Kamiar Rahnama Rad would like to thank NSF (National Science Foundation) for their generous support through grant number DMS-1810888.

\bibliography{references}

\appendix
\section{Technical Lemmas}
To improve the readability of the rest of the manuscript, we include several standard technical results used in our detailed proofs which are presented in the rest of the Appendix.

\begin{lemma}\label{lem:sum_xi_conc}
    Suppose Assumption A2 holds. Then for $p\geq 2$ we have
    \[
        \EE \|\bx_i\|^8\leq 24C_X
    \]
\end{lemma}
\begin{proof}
    Let $\bz=\bSigma^{-\frac12}\bx_i$ then $\bz\sim N(0,\bI_p)$ and $\|\bz\|^2\sim \chi^2(p)$. By standard results on $\chi^2$ distribution we have
    \[\EE \|\bz\|^8 = p(p+2)(p+4)(p+6)\].
    Then we have for $p\geq 2$:
    \begin{align*}
        \EE \|\bx_i\|^8 &= \EE (\bz^\top\bSigma\bz)^4\\
        &\leq \EE\left(\frac{C_X}{p}\|\bz\|^2\right)^4\\
        &\leq \frac{C_X^4}{p^4} \EE\|\bz\|^8\\
        &\leq C_X^4 \left(1+\frac2p\right)\left(1+\frac4p\right)\left(1+\frac6p\right)\\
        &\leq 24C_X^4.
    \end{align*}
\end{proof}
\begin{lemma}[\citep{rosenthal1970}]
    \label{lem:rosenthal}
    Let $\{X_i\}_{i\in [n]}$ be a sequence of independent non-negative random variables, and $t\geq 1$. Then we have
    \[
        \EE \left(\sum_{i\in [n]}X_i\right)^t
        \leq A(t) \max \left\{ \sum_{i\in[n]}\EE X_i^{t}, \left(\sum_{i\in[n]}\EE X_i\right)^{t} \right\}.
    \]
\end{lemma}
\begin{proof}
    See Theorem 3 in \citep{rosenthal1970}. For a discussion on sharp choice of $A(t)$, the readers may refer to \citep{ibragimov2001}.
\end{proof}

\section{Bounded Signal-to-noise Ratio}\label{subsec:bd_snr}
First we explain what we mean by 'bounded SNR' in Section \ref{subsec:assumptions}, using the three examples of linear, logistic and Poisson regression (with log exponential link):
\begin{itemize}
    \item Linear: $y_i|\bx_i \sim N(\bx_i^\top\bbeta^*, \sigma^2)$
    \item Logistic: $y_i|\bx_i\sim {\rm Binomial}((1+e^{-\bx_i^\top\bbeta^*})^{-1})$
    \item Poisson: $y_i|\bx_i\sim {\rm Poisson}(\log(1+e^{\bx_i^\top\bbeta^*}))$
\end{itemize}
Define the signal-to-noise ratio as
\[
    {\rm SNR}:= \frac{\var(\bx_i^\top\bbeta^*)}{\var(y_i|\bx_i^\top\bbeta^*)}.
\]
When $\|\bbeta^*\|_2 = O(\sqrt{p})$ or each elements of of $\bbeta^*$ is $O(1)$, by Assumption A2 we have
\[
    \var(\bx_i^\top\bbeta^*) = \bbeta^{*\top}\bSigma\bbeta^* \leq \frac{C_X}{p}\|\bbeta^*\|_2^2 = O(1).
\]
On the other hand, 
\begin{align*}
    [\var(y_i|\bx_i)]^{-1}=
    \begin{cases}
        \sigma^{-2}, & \text{ Linear }
        \\
        (e^{-\frac12 \bx_i^\top\bbeta^*} + e^{\frac12 \bx_i^\top\bbeta^*})^2,
        &\text{ Logistic}
        \\
        [\log(1+e^{\bx_i^\top\bbeta^*})]^{-1}, &\text{ Poisson}
    \end{cases}
\end{align*}
They are all $O_p(1)$ when $n,p$ increase. To see this, notice that 
\[
    \bx_i^\top\bbeta^*\sim N(0,\bbeta^{*\top}\bSigma\bbeta^*)
\]
and
\begin{align*}
    &4\leq (e^{-\frac12 z} + e^{\frac12 z})^2\leq 2(e^{|z|}+1),
    \\
    &(\log(2)+z_+)^{-1}\leq [\log(1+e^z)]^{-1}\leq z_+^{-1}
\end{align*}
where $z_+=z$ when $z\geq 0$ and $0$ otherwise. Suppose $C_1\leq p^{-1/2}\|\bbeta^*\|\leq C_2$ for some constants $C_1,C_2$, then $\bx_i^\top\bbeta^* = \Theta_p(1)$ and so is the ratio $\frac{\var(\bx_i^\top\bbeta^*)}{\var(y_i|\bx_i^\top\bbeta^*)}$.

\section{Discussion of the assumptions}
\subsection{On Assumption A4}\label{subsec:assumption_a4}
In this subsection we present several commonly used regularizers in machine learning, and show that they are all Lipshitz continuous. \\

\begin{example}[LASSO]
    The classic LASSO penalty is $r_0(\bbeta)=\|\bbeta\|_1$ and is clearly non-negative, convex and Lipschitz continuous. 
\end{example}

\medskip

\begin{example}[Group LASSO]
    The group LASSO was introduced in \citep{yuan_group_lasso} to achieve joint variable selection among different data groups. Assume the features are partitioned in $G$ groups $J_1, J_2, \ldots, J_G$, and the penalty takes the form $r_0(\bbeta)=\sum_{j=1}^G(\bbeta_{J_j}^\top \bK_{j}\bbeta_{J_j})^{1/2}$ where $\bbeta_{J_j}\in \RR^{p_j}$ is the coefficient vector for group $j$,  $\bK_j\in \RR^{p_j\times p_j}$ is a positive definite matrix, $\bbeta = (\bbeta_{J_1}^\top,\cdots,\bbeta_{J_G}^\top)^\top \in \RR^p$ is the concatenated coefficient with $p=\sum_{j=1}^G p_j$. 
    
    Clearly $r_0$ is non-negative. Moreover it is also convex since it is a sum of convex functions $r_{0j}(\bbeta)=(\bbeta_{J_j}^{\top}\bK_j\bbeta_{J_j})^{1/2}$. Finally, $r_0$ is also $\sqrt{\sum_{k=1}^J \sigma_{\max}(K_j)}$-Lipschitz with respect to $\bbeta$.
\end{example}

\medskip

\begin{example}[Generalized LASSO]
    The generalized LASSO penalty takes the form of $r_0(\bbeta)=\|\bD\bbeta\|_1$ and encompasses many LASSO type penalty such as LASSO ($\bD=\bI_p$) and fused LASSO ($\bD=(d_{ij})_{(p-1)\times p}$ where $d_{ij}=1$ if $i=j$, $d_{ij}=-1$ if $j=i+1$ and $d_{ij}=0$ otherwise). The nonnegativity and convexity are immediate since $r_0$ is the $\ell_1$ norm of $\bD\bbeta$. Moreover, it follows that $\|\bD\bbeta\|_1$ is $\sigma_{\max}(\bD)$-Lipschitz in $\bbeta$.
\end{example}

\medskip

\begin{example}[Nuclear norm and Schatten norms]
    When the estimand is a matrix $\bB$ with rank $d$, the nuclear norm is a popular regularizer:
    \[
        r_0(\bB) = \sum_{i=1}^d \sigma_i(\bB)
    \]
    where $\sigma_i(\bB)$ is the $i$-th largest singular value of $\bB$.
    More generally, the Schatten norm of $\bB$ is defined as
    \[
        r_0(\bB) = \left(\sum_{i=1}^{d} \sigma_i^p(\bB) \right)^{\frac1p}
    \]
    and takes the nuclear norm as a special case when $p=1$. Since $r_0(\bB)$ is a norm on the singular values of $\bB$, we obtain the nonnegativity and convexity of $r_0$. To show the Lipschitz continuity we argue as follows. Since $r_0(\bB)$ is the p-norm of the vector $(\sigma_1(\bB),\cdots,\sigma_d(\bB))$ and using the triangular inequality we have
    \begin{align*}
        |r_0(\bx)-r_0(\by)|\leq&~ r_0(\bx-\by) \\
        \leq&
        \begin{cases}
            K^{\frac{1}{p}-\frac{1}{2}}\norm{\bx-\by}_2\quad&\text{if }1\le p\le 2\\
           \norm{\bx-\by}_2\quad&\text{if }p\ge 2
        \end{cases}\\
        \le&~\sqrt{K}\norm{\bx-\by}_2,
    \end{align*}
    where $\norm{\bB}_2$ denotes the Frobenius norm of $\bB$. The second line uses the relationship between p-norms: for $\bu\in\RR^p$,
    \begin{align*}
        \|\bu\|_p 
        &= \left(\sum_{k=1}^K u_k^p\right)^{\frac1p}\\
        &\leq \begin{cases}
            K^{\frac{1}{p}-\frac{1}{2}}\norm{\bu}_2\quad&\text{if }1\le p\le 2\\
           \norm{\bu}_2\quad&\text{if }p\ge 2.
        \end{cases}
    \end{align*}
    Then if we let $\bu = (\sigma_1(\bB),\cdots,\sigma_K(\bB))^\top$ we then have $\|\bu\|_p = r_0(\bB)$ and 
    \[
        \|\bu\|_2 = \left(\sum_{k=1}^K \sigma_k^p(\bB)\right)^{\frac12} = \sqrt{\tr(\bB^\top\bB)} = \|\bB\|_2.
    \]
\end{example}
\subsection{On Assumption A5}\label{subsec:assumption_a5}
In this section, we show that the moment bound of $y_i$ and polynomial growth of $\ell,\dl,\phi,\dphi$ in Assumption A5 are justified for many popular data generating mechanisms. We show this for three examples: linear, logistic and Poisson regression (with log exponential link). In this subsection we assume Assumptions A2 and A3 hold, namely, $n/p \equiv\gamma_0>0$ and $\bx_i$ are i.i.d. $N(0,\bSigma)$ with $\sigma_{\max}(\bSigma)\leq C_X/p$. In addition we assume $p^{-1/2}\|\bbeta^*\|\leq \xi$ for some $\xi>0$. Finally, both $\ell$ and $\phi$ are set to the negative log-likelihood.\\

\begin{example}[Linear]
    Suppose $y_i|\bx_i \sim N(\bx_i^\top\bbeta^*,\sigma^2)$. Then $y_i \sim N(0, \tilde{\sigma}^2)$ where $\tilde{\sigma}^2=\sigma^2 + \bbeta^{*\top}\bSigma\bbeta^*$. 
    Using standard results on Gaussian moments, we have
    $$
    \EE|y|^m
    =\tilde{\sigma}^m\times
    \dfrac{2^{m/2}\Gamma\left(
    \frac{m+1}{2}
    \right)}{\sqrt{\pi}}
    \le 
    \left(
    \sigma^2+C_X\xi^2
    \right)^{m/2}    \dfrac{2^{m/2}\Gamma\left(
    \frac{m+1}{2}
    \right)}{\sqrt{\pi}}    
    \,\,
    \text{for any }
    m\ge 0.
    $$
    For Gaussian linear model, the negative log-likelihood is known to be the $\ell_2$ loss:
    \[
        \ell(y,z) = \frac12 (y-z)^2 \leq y^2+z^2,
    \]
    and $|\dl(y,z)|=|z-y|\leq |z|+|y|$. 
\end{example}

\begin{example}[Logistic]
    Suppose $y_i|\bx_i\sim {\rm Bernoulli}(p)$ where $p=(1+e^{-\bx_i^\top\bbeta^*})^{-1}$. Then $y_i\in \{0,1\}$ and obviously
    \[
        \EE|y|^m\le 1
        \,\,
        \text{for all }
        m\ge 0.
    \]
    The negative log-likelihood of this model is
    \[
        \ell(y,z) = y \log(1-e^{-z}) + (1-y)\log(1+e^z), \quad y \in\{0,1\},
    \]
    and
    \[
        |\ell(y,z)|\leq 2\log(2) + 2|z|.
    \]
    For its derivative,
    \[
        |\dl(y,z)| = \abs{\frac{e^z}{1+e^z} - y}\leq 1+|y|. \\
    \]
\end{example}

\begin{example}[Poisson]
    Suppose $y_i|\bx_i\sim {\rm Poisson}(\mu)$ with $\mu = \log(1+e^{\bx_i^\top\bbeta^*})$. Then $y_i$ are non-negative. 
    Using the equivalence between Poisson and exponential distributions, it can be shown that $y\sim {\rm Poisson}(\mu)$ implies $y\sim {\rm Sub-exp}(\mu)$. It then follows by results for sub-exponential random variables (see, e.g., Proposition 2.7.1 of \citep{vershynin2018high}) that
    $$
    \EE\left( |y|^m|\bx_i\right)\le (C_*\mu)^mm^m
    $$
    whence taking expectation over $\bx$ it follows that for any $m\ge 0$ we have
    $$
    \EE y^m\le \EE(\log(1+e^{\bx^{\top}\bbeta^*}))^m\times(C_*m)^m
    \le (C_*m)^m\EE e^{m\bx^{\top}\bbeta^*}
    \le (C_*m)^me^{\bbeta^{*\top}\bSigma\bbeta^*m^2/2}
    \le (C_*m)^me^{m^2C_X\xi^2/2}.
    $$
    for a numerical constant $C_*>0$. In the last series of inequalities we have used first the fact that $\log(1+x)\le x$ for $x>0$. Next we have used the moment generating function of the normal distribution $\bx^{\top}\bbeta^*\sim N(0,\bbeta^{*\top}\bSigma\bbeta^*)$ and finally the inequality $\bbeta^{*\top}\bSigma\bbeta\le C_X\xi^2$ in the last step.

    The negative log-likelihood is 
    \begin{align*}
        \ell(y,z)| &= \log(y!) + \log(1+e^z) - y \log\log(1+e^z)\\
        &\leq |y\log(y)| + \log(2)+|z| + |y \log(\log(2)+|z|)|\\
        &\leq y^2 + \log(2)+|z| + |y(\log\log(2)+\frac{z}{\log(2)})| \\
        &\leq y^2 + \log(2)+|z|+ \log\log(2)|y| + \frac{1}{2\log(2)}( y^2+z^2)\\
        &\leq C(y^2+z^2+1),
    \end{align*}
    and the derivative satisfies
    \[
        |\dl(y,z)| = \abs{
        \frac{1}{1+e^{-z}} - \frac{ye^z}{(1+e^z)\log(1+e^z)}}\leq 1+|y|.
    \]
\end{example}


\section{Detailed Proofs}

\subsection{Proof of Lemma \ref{lem:gaussian_smooth} }
\label{subsec: approx_r}


We start with the proof of Part (a). First we extend $r_0$ to $\RR^p$ by simply letting $r_0(\bbeta)=0$ for $\bbeta \notin \bTheta$. Then we have
    \begin{align*}
        r_0^\alpha(\bbeta)& = \int_{\RR^p}r_0(\bbeta-\bz)\alpha\phi(\alpha\bz)d\bz \\
        &= \alpha \int_{\RR^p} r_0(\bz) \phi(\alpha(\bbeta-\bz))d\bz.
    \end{align*}
    For a point $\bbeta\in\bTheta_0$, consider its directional derivative over a direction $\bv\in \RR^p$ with $\norm{\bv}_2=1$:
    \begin{align*}
        &~\nabla_\bv r_0^\alpha(\bbeta)\\
        &:= \lim_{h\to 0} \frac1h [r_0^\alpha(\bbeta+h\bv)-r_0^\alpha(\bbeta)]\\
        & \leq 2 \alpha \nabla_\bv \phi(\alpha(\bbeta-\bz))\\
        & = \lim_{h\to 0} \alpha \int_{\RR^p} r_0(\bz)\frac1h 
        [ \phi(\alpha(\bbeta-\bz + h\bv)) - \phi(\alpha(\bbeta-\bz)) ] d\bz
    \end{align*}
    Since $\phi \in \cC^\infty(\RR^p)$, we have that for small enough $h$:
    \begin{align*}
        &~\frac1h [ \phi(\alpha(\bbeta-\bz + h\bv)) - \phi(\alpha(\bbeta-\bz)) ]\\
        &\leq  2\alpha \bv^\top \nabla \phi(\alpha(\bbeta-\bz))\\
        &\leq  2\alpha \norm{\nabla \phi(\alpha(\bbeta-\bz))}\\
        &\leq  2\alpha  (2\pi)^{-p/2} e^{-\frac12 \alpha^2 \norm{\bbeta-\bz}_2^2} \norm{\bbeta-\bz}_2.
    \end{align*}
    By assumption $r_0(\bz)e^{-\frac12 \alpha^2 \norm{\bbeta-\bz}_2^2} \norm{\bbeta-\bz}_2$ is integrable, then using Dominated Convergence, the limit can be taken within the integral, so that $\nabla_\bv r_0^\alpha(\bbeta)$ exists. Using similar arguments, in order that $\nabla_\bv^k r_0^\alpha(\bbeta)$ exists, we only need $r_0(\bz)\nabla_\bv^k \phi(\alpha(\bbeta-\bz)) $ to be integrable. To see this, notice that
    \[
        \frac{\partial^k}{\partial u_1^{k_1}\cdots\partial u_p^{k_p}}\phi(\bu)
        = e^{-\frac12\norm{\bu}^2} P_{k_1\cdots k_p}(\bu)
    \]
    where $P_{k_1\cdots k_p}(\bu) = u_1^{k_1}u_2^{k_2}\cdots u_p^{k_p} + o(u_1^{k_1}u_2^{k_2}\cdots u_p^{k_p})$ is a polynomial with its dominating term being $u_1^{k_1}u_2^{k_2}\cdots u_p^{k_p}$ with order $K$. Then we have
    \begin{align*}
        &~\abs{\nabla_\bv^k \phi(\bu)}\\
        &= \abs{\sum_{k_1 + \cdots + k_p=k}\frac{\partial^k}{\partial u_1^{k_1}\cdots\partial u_p^{k_p}}\phi(\bu)v_1^{k_1}\cdots v_p^{k_p}}\\
        &\leq \sum_{k_1 + \cdots + k_p=k}\abs{v_1^{k_1}\cdots v_p^{k_p}} e^{-\frac12\norm{\bu}^2} P_{k_1\cdots k_p}(\bu)\\
        &\leq \left(\sum_{k_1 + \cdots + k_p=k} \abs{P_{k_1\cdots k_p}(\bu)} \right) e^{-\frac12\norm{\bu}^2} \left(\sum |v_i|\right)^k\\
        &= O(\norm{\bu}^k e^{-\frac12\norm{\bu}^2}p^{k/2}\norm{\bv}_2^k )\\
        &= O(\norm{\bu}^k e^{-\frac12\norm{\bu}^2})
    \end{align*}
    Inserting $\bu = \alpha(\bbeta-\bz)$ we have
    \begin{align*}
        \abs{\nabla_\bv^k \phi(\alpha(\bbeta-\bz))}
        &= O(\alpha^k\norm{\bbeta-\bz}_2^k e^{-\frac12 \alpha^2 \norm{\bbeta-\bz}_2^2 })\\
        &= O(\norm{\bz}_2^k e^{-\frac12 \alpha^2 \norm{\bbeta-\bz}_2^2})
    \end{align*}
    So $r_0(\bz)\nabla_\bv^k \phi(\alpha(\bbeta-\bz)) $ is integrable if $r_0(\bz) \norm{\bz}_2^k e^{-\frac12 \norm{\bbeta-\bz}_2^2}$ is integrable. The continuity of the derivatives follows then from exchanging the integration and differentiation.

We now turn tho the proof of Part (b) of the paper. For all $\bbeta\in \bTheta_0$,
    \begin{align*}
        |r_0^\alpha(\bbeta) - r_0(\bbeta)|&\leq \int_{\RR^p} |r_0(\bbeta-\bz)-r_0(\bbeta)| \alpha\phi(\alpha\bz)d\bz \\
        & = \int_{\RR^p} |r_0(\bbeta - \alpha^{-1}\bu)-r_0(\bbeta)| d\Phi(\bu)\\
        &\leq \frac{L}{\alpha^k} \int_{\RR^p} \norm{\bu}_2^k d\Phi(\bu)\\
        &\to 0 \quad(\alpha\to\infty)
    \end{align*}
    The second line uses change of variable $\bu = \alpha\bz$ and the third line uses our assumption on $r_0$. The last line is because the integral is finite.
\qed

\begin{remark}
    Although Lemma \ref{lem:gaussian_smooth} adopts a weaker assumption than Lipschitz continuity, it can be shown that $k=1$ is the only value that allows $r_0$ to be convex. In fact, for $k<1$ we have 
    \[
        r_0(\bx)\leq L\|\bx\|^k.
    \]
    It is obvious that unless $r_0$ is constant, it would eventually increase at least linearly thus cannot be bounded by $L\|\bx\|^k$ for some $k<1$.

\end{remark}

\begin{example}
    Generalized LASSO: $r_0(\bbeta) = \norm{\bD\bbeta}_1$ for a fixed $\bD\in\RR^{m\times p}$. It is continuous, and $|r_0(\bx)-r_0(\by)| = |\norm{\bD\bx}_1 - \norm{\bD\by}_1|\leq \norm{\bD\bx-\bD\by}_1\leq \sqrt{m}\norm{\bx-\by}_2$.
    Note that the classical LASSO is a special case with $\bD=\bI_p$.
\end{example}

\begin{remark}
    In fact if $r_0$ is a norm such that $\EE r_0(\bbeta)<\infty$ where $\bbeta$ has i.i.d. $N(0,1)$ entries, then the results of the above two lemmas hold. 
\end{remark}

\subsection{Proof of Lemma \ref{lem:proj_property}}
\label{subsec:proof_proj_property}
    Fix $\bx,\by$ and consider $f(t):= \bPi_{\bTheta}((1-t)\bx+t\by)$. By the firm non-expansiveness of projection operators, we have $\forall t_1,t_2\in[0,1]$,
    \begin{align*}
        0&\leq (f(t_1)-f(t_2))^2
        \\
        =&\|\bPi_{\bTheta}((1-t_1)\bx+t_1\by) - \bPi_{\bTheta}((1-t_2)\bx+t_2\by)\|^2
        \\
        \leq&~\langle \bPi_{\bTheta}((1-t_1)\bx+t_1\by) - \bPi_{\bTheta}((1-t_2)\bx+t_2\by), \nonumber  (t_1 - t_2 )(\by-\bx)\rangle
        \\
        =&\langle f(t_1) - f(t_2), (t_1 - t_2 )(\by-\bx) \rangle 
        \\
        \leq & (t_1-t_2)^2\|\by-\bx\|^2 \label{eq:proj_non_expansive}.\numberthis
    \end{align*}

    By taking a square root, we thus have $f(t)$ is $\|\bx-\by\|$ Lipschitz on $[0,1]$. Therefore it is absolutely continuous and differentiable almost everywhere. Let $t_1 = t+\epsilon$ and $t_2=t$, and then divide 
    \eqref{eq:proj_non_expansive} by $\epsilon^2$, then if $f'(t)$ exists we have:
    \begin{align*}
        0\leq\|f'(t)\|^2\leq\langle f'(t), \by-\bx\rangle\leq \|\by-\bx\|^2
    \end{align*}
    For each $t$, consider solving the following equation for $\bJ(t)$:
    \[
        f'(t) = \bJ(t)(\by-\bx)\label{eq:f_prime_J}\numberthis.
    \]
    The claim is that we can find $\bJ(t)$ with its eigenvalues between $[0,1]$ such that \eqref{eq:f_prime_J} holds. In fact, since the rotation in $\RR^p$ is a unitary matrix with all eigenvalues being $1$, we can choose $\bJ(t)=\frac{\|f'(t)\|}{\|\by-\bx\|}\bR$ where $\bR$ is a rotation that rotates $\by-\bx$ to the direction of $f'(t)$, and this choice of $\bJ(t)$ has its all eigenvalues being $\frac{\|f'(t)\|}{\|\by-\bx\|}\in [0,1]$.

    With a slight abuse of notation we assume $f'(t)=\mathbf{0}$ wherever the derivative does not exist. Then by the Newton-Leibniz formula (for Lebesgue integral and absolute continuous functions):
    \begin{align*}
        \bPi_{\bTheta}(\by)-\bPi_{\bTheta}(\bx) 
        &= f(1)-f(0)\\
        &= \int_0^1 f'(t)dt\\
        &= \int_0^1 \bJ(t) (\by-\bx)dt
    \end{align*}
    where the last line uses \eqref{eq:f_prime_J}.\qed

\begin{remark}
In the proof of Lemma \ref{lem:proj_property}, the matrix $\bJ(t)$ is not explicitly derived. However by equation \eqref{eq:f_prime_J} one can easily see that, when $\bPi_{\bTheta}((1-t)\bx+t\by)$ is smooth at $t=t_0$, we can use its Jacobian as $\bJ(t)$, i.e. 
\[
    \bJ(t_0) = \left.\frac{\partial}{\partial \bz} \bPi_{\bTheta}(\bz)\right|_{\bz = (1-t_0)\bx+t_0\by}.
\]
\end{remark}

\subsection{Proof of Lemma \ref{lem:beta_lo_error}}
\label{subsec:proof_beta_lo_error}

\begin{proof}
Begin we start the proof, let us introduce the following notations:
    \begin{align*}
        h(\bbeta)&:=\sum_{j=1}^n \ell_j(\bbeta) + \lambda (1-\eta) r_0(\bbeta) + \lambda \eta \bbeta^\top\bbeta,\\
        h^\alpha(\bbeta)&:=\sum_{j=1}^n \ell_j(\bbeta) + \lambda (1-\eta) r_0^\alpha(\bbeta) + \lambda \eta \bbeta^\top\bbeta, \\
                h_{/ i}(\bbeta)&:=\sum_{j\neq i}^n \ell_j(\bbeta) + \lambda (1-\eta) r_0(\bbeta) + \lambda \eta \bbeta^\top\bbeta,\\
        h_{/i}^\alpha(\bbeta)&:=\sum_{j\neq i}^n \ell_j(\bbeta) + \lambda (1-\eta) r_0^\alpha(\bbeta) + \lambda \eta \bbeta^\top\bbeta.
    \end{align*}
As mentioned before, in order to obtain a bound on $\|\hbbeta- \hbbeta_{/i}\|_2$ we use the smoothing trick. Hence, we first obtain an upper bound for $\|\hbbeta^\alpha - \hbbeta_{/i}^\alpha\|$. 
    
    Similar to $h$ and $h^\alpha$, let $h_{/i}$ and $h_{/i}^\alpha$ denote the loss functions for $\hbbeta_{/i}$ and $\hbbeta_{/i}^\alpha$ respectively. By Lemma \ref{lem:fix_eq_prox}, $\hbbeta^\alpha$ and $\hbbeta_{/i}^\alpha$  satisfy
    \begin{align*}
        \hbbeta^\alpha = \bPi_{\bTheta}(\hbbeta^\alpha-\nabla h^{\alpha}(\hbbeta^\alpha))\\
        \hbbeta_{/i}^\alpha = \bPi_{\bTheta}(\hbbeta_{/i}^\alpha-\nabla h_{/i}^{\alpha}(\hbbeta_{/i}^\alpha)).
    \end{align*}
    By subtracting one from the other we have
    \begin{align*}
        \hbbeta^\alpha - \hbbeta_{/i}^\alpha 
        =& \bPi_{\bTheta}(\hbbeta^\alpha-\nabla h^{\alpha}(\hbbeta^\alpha)) - \bPi_{\bTheta}(\hbbeta_{/i}^\alpha-\nabla h_{/i}^{\alpha}(\hbbeta_{/i}^\alpha))\\
        =& \bar{\bJ}
        \times \left(\hbbeta^\alpha-\nabla h^{\alpha}(\hbbeta^\alpha) - \hbbeta_{/i}^\alpha+\nabla h_{/i}^{\alpha}(\hbbeta_{/i}^\alpha)\right)
        \label{eq:diff_in_proj}\numberthis
    \end{align*}
    where the second line comes from Lemma \ref{lem:proj_property} and 
    $$\bar{\bJ}:=\int_0^1  \bJ(t) dt.$$ It is straightforward to use Lemma \ref{lem:proj_property} to show that
    \[
        0\leq \lambda_{\min} (\bar{\bJ}) \leq\lambda_{\max} (\bar{\bJ}) \leq 1.
    \]
    On the other hand,
    \begin{align*}
        &\hbbeta^\alpha-\nabla h^{\alpha}(\hbbeta^\alpha) - \hbbeta_{/i}^\alpha+\nabla h_{/i}^{\alpha}(\hbbeta_{/i}^\alpha)\\
        =& \hbbeta^\alpha - \hbbeta_{/i}^\alpha - \sum_{j\in [n]} \left[\dl_j(\hbbeta^\alpha) - \dl_j(\hbbeta_{/i}^\alpha)\right]\bx_j - \lambda \left[\nabla r^\alpha(\hbbeta^\alpha) -  \nabla r^\alpha(\hbbeta_{/i}^\alpha) \right] - 
        \dl_i(\hbbeta_{/i}^\alpha)\bx_i\\
        =&  \left( \bI - \bX^\top\diag[\ddl_j(\bxi_j)]_{j\in [n]}\bX - \lambda \nabla^2 r^\alpha(\bXi) \right)  (\hbbeta^\alpha - \hbbeta_{/i}^\alpha) - \dl_i(\hbbeta_{/i}^\alpha)\bx_i\label{eq:lem_betaloerror_diff}\numberthis
    \end{align*}
    where the second line uses the definition of $h^\alpha$ and the third line uses mean-value-theorem on $\dl_j$ and $\nabla r^\alpha$:  

    \[
        \ddl_j(\bxi_j) :=\int_0^1 \ddl_j(t\hbbeta^\alpha +(1-t)\hbbeta_{/i}^\alpha)dt
    \]
    with 
    \[
        \dl_j(\hbbeta^\alpha) - \dl_j(\hbbeta_{/i}^\alpha) = \ddl_j(\bxi_j) \bx_j^\top (\hbbeta^\alpha - \hbbeta_{/i}^\alpha)
    \]
    and likewise 

    \[
        \nabla^2 r^\alpha(\bXi) := \int_0^1 \nabla^2 r^\alpha( t \hbbeta^\alpha + (1-t) \hbbeta_{/i}^\alpha )dt
    \]
    with
    \[
        \nabla r^\alpha(\hbbeta^\alpha) -  \nabla r^\alpha(\hbbeta_{/i}^\alpha) = \nabla^2 r^\alpha(\bXi) (\hbbeta^\alpha - \hbbeta_{/i}^\alpha).
    \]
    Inserting this back to \eqref{eq:diff_in_proj} we have 

    \[
        \hbbeta^\alpha - \hbbeta_{/i}^\alpha = -\bG^{-1}\bar{\bJ}\left(
        \dl_i(\hbbeta_{/i}^\alpha)\bx_i\right)
    \]
    where 
    \begin{align*}
        \bG &:=  \bI + \bar{\bJ}\left(  \bX^\top\diag[\ddl_j(\bxi_j)]_{j\in[n]}\bX + \lambda \nabla^2 r^\alpha(\bXi) - \bI \right)\\
        & = \bI - \bar{\bJ} + \bar{\bJ}\left( \bX_{/i}^\top\diag[\ddl_j(\bxi_j)]_{j\neq i}\bX_{/i} + \lambda \nabla^2 r^\alpha(\bXi) \right).
    \end{align*}
    Since $\bI - \bar{\bJ}$ is positive semidefinite (note that all the eigenvalues of $\bar{\bJ}$ are in $[0,1]$ ) and  $\bX^\top\diag[\ddl_j(\bxi_j)]_{j\in[n]}\bX + \lambda \nabla^2 r^\alpha(\bXi)$ is positive definite (due to the ridge component), $\bG$ is also positive definite.
    Hence, we have
    \[
        \norm{\hbbeta^\alpha - \hbbeta_{/i}^\alpha}_2 \leq \left[\sigma_{\min}(\bG)\right]^{-1}\sigma_{\max}(\bar{\bJ})
        \left\Vert
        \dl_i(\hbbeta_{/i}^\alpha)\bx_i\right\Vert_2.
        \label{eq:normz_diff_beta_lo}\numberthis
    \]
    We have already established $\sigma_{\max}(\bar{\bJ})\leq 1$. Now we bound $\sigma_{\min}(\bG)$:
    \begin{align*}
        \sigma_{\min}(\bG) &= \lambda_{\min}(\bG)\\
        &= 1+ \lambda_{\min}\left( \bar{\bJ}\left(  \bX_{/i}^\top\diag[\ddl_j(\bxi_j)]_{j\notin I_i}\bX_{/i} + \lambda \nabla^2 r^\alpha(\bXi) - \bI \right) \right)\\
        &:= 1+\lambda_{\min}(\bar{\bJ} \bM)
    \end{align*}
    where $\bM := \bX_{/i}^\top\diag[\ddl_j(\bxi_j)]_{j\notin I_i}\bX_{/i} + \lambda \nabla^2 r^\alpha(\bXi) - \bI$. Observe that due to the existence of the ridge component in $r$, we have 
    \[\lambda_{\min}(\bM)\geq 2\lambda \eta - 1.\]
    \begin{itemize}
        \item If $\lambda_{\min}(\bM)\geq 0$, then we have
        \[
            \lambda_{\min}\left( \bar{\bJ}\bM\right) \geq \lambda_{\min}(\bar{\bJ})\lambda_{\min}(\bM)\geq 0.
        \]
        \item If $2\lambda\eta-1\leq \lambda_{\min}(\bM)<0$, then we have
        \[
            \lambda_{\min}\left( \bar{\bJ}\bM \right)\geq \lambda_{\max}(\bar{\bJ})\lambda_{\min}(\bM)\geq 2\lambda\eta -1.
        \]
    \end{itemize}
    Therefore we have
    \[
        \sigma_{\min}(\bG) \geq 1 + (2\lambda\eta -1)\wedge 0 = 2\lambda\eta \wedge 1.
    \]
    inserting this back to \eqref{eq:normz_diff_beta_lo} we have
    \[
        \norm{\hbbeta^\alpha - \hbbeta_{/i}^\alpha}\leq \frac{\norm{\dl_i(\hbbeta_{/i}^\alpha)\bx_i}}{2\lambda\eta\wedge 1}.
    \]
The next step of the proof is to use this upper bound on the smoothed estimates and obtain an upper bound for $\norm{\hbbeta - \hbbeta_{/i}}$. Towards this goal, we first prove the following lemma:  

\begin{lemma}\label{lem:beta_smooth_error}
    Under assumptions A1-A4, we have that
    \[
            \norm{\hbbeta^\alpha-\hbbeta} \leq \sqrt{\frac{2(1-\eta)}{\eta}\norm{r_0^\alpha - r_0}_\infty},
    \]
     and similarly
    \[
        \norm{\hbbeta^\alpha_{/i}-\hbbeta_{/i}} \leq \sqrt{\frac{2(1-\eta)}{\eta}\norm{r_0^\alpha - r_0}_\infty}.
    \]
\end{lemma}
\begin{proof}
 
    We have
    \begin{align*}\label{eq:h-alfa-diff}
        & h^\alpha(\hbbeta)-h^\alpha(\hbbeta^\alpha) = h^\alpha(\hbbeta)-h(\hbbeta^\alpha)+h(\hbbeta^\alpha)-h^\alpha(\hbbeta^\alpha)\\
        & \le  h^\alpha(\hbbeta)-h(\hbbeta)+h(\hbbeta^\alpha)-h^\alpha(\hbbeta^\alpha) \le  2\lambda(1-\eta) \norm{r_0^\alpha- r_0}_\infty\numberthis
    \end{align*}
    where the first inequality uses the fact that $\hbbeta,\hbbeta^\alpha$ are the minimizers of $h(\bbeta)$ and $h^\alpha(\bbeta)$ respectively, and the last inequality uses the definition of $h, h^\alpha$. On the other hand we also have from the second order mean value theorem (or Taylor expansion) that
    \begin{align*}\label{eq:h-alfa-diff2}
        &h^\alpha(\hbbeta)-h^\alpha(\hbbeta^\alpha)\\
        =& \nabla h^\alpha(\hbbeta^\alpha)^{\top}(\hbbeta-\hbbeta^\alpha)
        +\frac{1}{2}(\hbbeta-\hbbeta^\alpha)^\top
        \nabla^2 h^{\alpha}(\bXi)
        (\hat{\bbeta}-\hat{\bbeta}^{\alpha})\\
        \ge& \frac12 (\hat{\bbeta}-\hat{\bbeta}^{\alpha})^{\top}
        \nabla^2h^{\alpha}(\bXi)
        (\hat{\bbeta}-\hat{\bbeta}^{\alpha})\\
        \ge&~\lambda\eta\|\hbbeta-\hbbeta^\alpha\|^2.\numberthis
    \end{align*}
    In the first equality, by a slight abuse of the notation, we have written $\nabla^2h^\alpha(\bXi)$ as the Hessian in the second order Taylor expansion.\footnote{Since the mean-value theorem doesn't exist for vector-valued functions, the matrix $\nabla^2h^\alpha(\bXi)$ is actually the Hessian of $h^\alpha$, with each row evaluated at a different convex combination of $\hbbeta$ and $\hbbeta^\alpha$, as the matrix $\bXi$ indicates.}
    To obtain the second line, observe that 
    \[
    \nabla h_{\alpha}(\hat{\bbeta}^{\alpha})^\top (\hbbeta-\hbbeta^\alpha) = \left.\frac{\partial}{\partial t} h^\alpha((1-t)\hbbeta^\alpha + t\hbbeta)\right|_{t=0}.
    \]
    Since $\hbbeta^\alpha$ is the minimizer of $h^\alpha$, $t=0$ is the minimizer of $h^\alpha((1-t)\hbbeta^\alpha + t\hbbeta)$ for $t\in [0,1]$. And since $h^\alpha$ is smooth on $(0,1)$ we must have $\nabla h_{\alpha}(\hat{\bbeta}^{\alpha})^\top (\hbbeta-\hbbeta^\alpha) = \left.\frac{\partial}{\partial t} h^\alpha((1-t)\hbbeta^\alpha + t\hbbeta)\right|_{t=0}\geq0$. The last line of \eqref{eq:h-alfa-diff2} is because $\nabla^2h^\alpha(\bXi)$ is positive-definite and $\sigma_{\min}(\nabla^2h^\alpha(\bXi))\geq 2\lambda\eta$  due to the existence of the ridge component.

    Combining \eqref{eq:h-alfa-diff} and \eqref{eq:h-alfa-diff2} we have
    \begin{align*}
        \norm{\hbbeta-\hbbeta^\alpha}^2 \leq \frac{2(1-\eta)}{\eta}\norm{r_0^\alpha-r_0}_\infty.
    \end{align*}
    Using a similar argument we can also prove that 
    \[
        \norm{\hbbeta_{/i}-\hbbeta_{/i}^\alpha}^2 \leq \frac{2(1-\eta)}{\eta}\norm{r_0^\alpha-r_0}_\infty.
    \]
\end{proof}
From lemma \ref{lem:gaussian_smooth} we can see that $\norm{r_0^\alpha-r_0}_\infty \rightarrow 0$ as $\alpha \rightarrow \infty$. 

Continuing with the proof of Lemma \ref{lem:beta_lo_error}, by Lemma \ref{lem:beta_smooth_error}, $\hbbeta^\alpha\to \hbbeta$ and $\hbbeta_{/i}^\alpha \to \hbbeta_{/i}$ when $\alpha\to\infty$. Hence, we can let $\alpha\to\infty$ and have
\[
    \norm{\hbbeta - \hbbeta_{/i}}\leq \frac{\norm{\dl_i(\hbbeta_{/i})\bx_i}}{2\lambda\eta\wedge 1}.
\]
\end{proof}

\subsection{Proof of Lemma \ref{lem:mean_dphi0}}\label{subsec:proof_mean_dphi0}
We first obtain an upper bound for the $4^{\rm th}$ moment. Using Assumption A5, we have

\begin{align*}
    \sqrt{\EE\dphi_0^4(\bbeta)}
    &\leq \sqrt{\EE \left( 1+|y_0|^s + |\bx_0^\top \bbeta|^s  \right)^4}
    \\
    &\leq \sqrt{27(1+\EE|y_0|^{4s}+\EE|\bx_0^\top \bbeta|^{4s})}
    \\
    &\leq \sqrt{ 27\left(1+C_Y(4s)+(4s-1)!!\left(\frac{C_X}{p}\|\bbeta\|^2\right)^{2s}\right) }
    \\
    &\leq \sqrt{27(1+C_Y(4s))} + \sqrt{27(4s-1)!!}C_X^s \left(\frac{C_X}{p}\|\bbeta\|^2\right)^s.
\end{align*}
where the second line uses the simple equation $(a+b+c)^4\leq 27a^4+27b^4+27c^4$, the third line uses  $\bx_0^\top \bbeta\sim N(0,\bbeta^\top\bSigma\bbeta)$ with $\bbeta^\top\bSigma\bbeta\leq \frac{C_X}{p}\|\bbeta\|^2$. It also uses the moment formula for standard Gaussian variable. The last line uses the equation $\sqrt{a+b}\leq \sqrt{a}+\sqrt{b}$ for $a,b\geq 0$.
Define $C_{\phi}^2 = \max\{ \sqrt{27(1+C_Y(4s))} , \sqrt{27(4s-1)!!}C_X^s  \}$, then
\begin{align*}
    \sqrt{\EE\dphi_0^4(\bbeta)} \leq C_{\phi}^2 \left(1+\left(\frac{C_X}{p}\|\bbeta\|^2\right)^s \right)
\end{align*}
For the second moment, notice that 
\begin{align*}
    \sqrt{\EE\dphi_0^2(\bbeta)} 
    &\leq [\EE\dphi_0^4(\bbeta)]^{\frac14}
    \\
    &\leq C_{\phi} \sqrt{1+\left(\frac{C_X}{p}\|\bbeta\|^2\right)^s }
    \\
    &\leq C_{\phi} \left(1+\left(\frac{C_X}{p}\|\bbeta\|^2\right)^{s/2}\right).
\end{align*}
where the last step uses $\sqrt{a+b}\leq \sqrt{a}+\sqrt{b}$ again. The same arguments lead to the same bound for $\sqrt{\EE\phi_0^2(\bbeta)} $.

\subsection{Proof of Lemma \ref{lem:assumption_a5}}\label{subsec:assn_a5}
\begin{proof}
\begin{enumerate}
    \item[(a)]
    Without loss of generality we assume $C=1$ in Assumption A5.
    Throughout this proof we use $h(\bbeta), h_{/i}(\bbeta), h^\alpha(\bbeta)$ and $h_{/i}^\alpha(\bbeta)$ for the loss functions of the corresponding models, defined as:
        \begin{align*}
        h(\bbeta)&:=\sum_{j=1}^n \ell_j(\bbeta) + \lambda (1-\eta) r_0(\bbeta) + \lambda \eta \bbeta^\top\bbeta,\\
        h^\alpha(\bbeta)&:=\sum_{j=1}^n \ell_j(\bbeta) + \lambda (1-\eta) r_0^\alpha(\bbeta) + \lambda \eta \bbeta^\top\bbeta, \\
                h_{/ i}(\bbeta)&:=\sum_{j\neq i}^n \ell_j(\bbeta) + \lambda (1-\eta) r_0(\bbeta) + \lambda \eta \bbeta^\top\bbeta,\\
        h_{/i}^\alpha(\bbeta)&:=\sum_{j\neq i}^n \ell_j(\bbeta) + \lambda (1-\eta) r_0^\alpha(\bbeta) + \lambda \eta \bbeta^\top\bbeta.
    \end{align*}
    It is straightforward to show that 
    \begin{align*}
        \lambda\eta \|\hbbeta\|_2^2\le&~\sum_{j\in[n]} \ell(y_j,\bx_j^\top\hbbeta) + \lambda(1-\eta)r_0(\hbbeta)+\lambda \eta \|\hbbeta\|_2^2
        \leq~ \sum_{j=1}^n \ell(y_j,0), 
    \end{align*}
    where the last inequality is due to the fact that $h(\hbbeta) \leq h(\mathbf{0})$ and that $\ell(y,z)\geq 0$. Similarly we have $\lambda\eta \|\hbbeta_{/i}\|_2^2\leq \sum_{j=1}^n \ell(y_j,0)$.

    We can then use Assumption A5 to obtain
    \begin{align*}
        \EE \left(\frac1p \|\hbbeta\|^2\right)^{t}
        \leq (p\lambda\eta)^{-t}\EE [n+|y_1|^s+\cdots+|y_n|^s]^{t}.
    \end{align*}
    By Rosenthal inequality (Lemma \ref{lem:rosenthal}),  we have
    \begin{align*}
        \EE \left(\frac1p \|\hbbeta\|^2\right)^{t}
        &\leq (p\lambda\eta)^{-t} A(t) \max\left\{ n^{t}+ n\EE|y_1|^{st}, (n+n\EE|y_1|^s)^{t} \right\}
        \\
        &\leq \left(\frac{\gamma_0}{\lambda\eta}\right)^{t}A(t)
        \max\left\{ 1+n^{1-t}C_Y(st), (1+C_Y(s))^{t} \right\}
        \\
        &\leq \left(\frac{\gamma_0}{\lambda\eta}\right)^{t}A(t)
        \max\left\{ 1+C_Y(st), (1+C_Y(s))^{t} \right\}
        := C_{\beta}(t)\label{eq:high_moment_bbeta}\numberthis,
    \end{align*}
    where $C_Y(\cdot)$ is the constant in Assumption A5, and $A(\cdot)$ is the Rosenthal constant in Lemma \ref{lem:rosenthal}.

    \item[(b)] Our next goal is to obtain an upper bound for $\EE\dl_1^8(\hbbeta_{/1})$. By using Assumption A5 we have
     \begin{align*}
        \EE\dl_1^8(\hbbeta_{/1})  
        &\leq \EE [1+|y_1|^s+|\bx_1^\top\hbbeta_{/1}|^s]^8
        \\
        &\leq 3^7 \left( 1+\EE|y_1|^{8s} + \EE|\bx_1^\top\hbbeta_{/1}|^{8s} \right).\numberthis
        \label{eq:moment_xTbeta_part1}
    \end{align*}

    Next we bound $\EE|\bx_1^\top\hbbeta_{/1}|^{8s}$:
    \begin{align*}
        \EE|\bx_1^\top\hbbeta_{/1}|^{8s}
        &= \EE \left[\EE[ |\bx_1^\top\hbbeta_{/1}|^{8s} |\hbbeta_{/1}]\right]
        \\
        &\leq \EE [(8s-1)!! (\hbbeta_{/1}^\top\bSigma\hbbeta_{/1})^{4s}]
        \\
        &\leq (8s-1)!!  \EE \left(\frac{C_X}{p}\|\hbbeta_{/1}\|^2\right)^{4s}
        \\
        &\leq (8s-1)!! C_X^{4s} C_{\beta}(4s)
    \end{align*}
    where the second line uses the fact that $\bx_1^\top\hbbeta_{/1}|\hbbeta_{/1} \sim N(0,\hbbeta_{/1}^\top\bSigma\hbbeta_{/1})$ and that the $t^{\rm th}$ moment of a standard Gaussian variable is $(t-1)!!$ whenever $t$ is a positive even number. The last line uses \eqref{eq:high_moment_bbeta}. Inserting this back to \eqref{eq:moment_xTbeta_part1} we have:
    \begin{align*}
        \EE\dl_1^8(\hbbeta_{/1})  
        \leq 3^7(1+C_Y(8s) + (8s-1)!!C_X^{4s}C_{\beta}(4s))
        :=C_{\ell}.
    \end{align*}
    \end{enumerate}
\end{proof}

\subsection{Proof of Lemma~\ref{lem:v1}}\label{sec:pf-lemv1}
Recall that
\[
    \LO = \frac1n \sum_{i=1}^n\phi_j(\hbbeta_{/i}).
\]

Then we have
\begin{align*}\label{eq:V1bd}
    \EE V_1^2=&\EE \left(\LO  - \frac{1}{n} \sum_{i=1}^n \EE[ \phi_i( \hbbeta_{/i}) |  \cD_{/i}] \right)^2 \\
    =& \frac{1}{n} \EE \left(\phi_1(\hbbeta_{/1}) - \EE [\phi_1(\hbbeta_{/1})|\cD_{/1}]\right)^2
    \\
    &+
    \frac{n-1}{n}\EE \left(\phi_1(\hbbeta_{/1}) - \EE [\phi_1(\hbbeta_{/1})|\cD_{/1}]\right)
    \cdot \left(\phi_{2}(\hbbeta_{/2}) - \EE [\phi_2(\hbbeta_{/2})|\cD_{/2}]\right).\numberthis
\end{align*}
Note that 
\begin{align*}
    &~\EE \left[\left(\phi_1(\hbbeta_{/1}) - \EE (\phi_1(\hbbeta_{/1})|\cD_{/1})\right)\right]^2\\
    =&~ \EE \var(\phi_1(\hbbeta_{/1})|\cD_{/1}) \\
    \leq&~ \EE \left(\EE[\phi_1^2(\hbbeta_{/1})|\cD_{/1}]\right)\\
    \leq&~ \EE \left( C_{\phi} + C_{\phi}\left( \frac1p \|\hbbeta_{/1}\|^2\right)^{\frac{s}{2}} \right)^2\\
    \leq&~ 2C_{\phi}^2 + 2C_{\phi}^2 \EE \left(\frac1p \|\hbbeta_{/1}\|^2\right)^{s}\\
    \leq&~ 2C_{\phi}^2\left(1+ C_{\beta}(s)\right):=C_{v1,1}\numberthis\label{eq:v1_first_term}
\end{align*}
where the fourth line uses Lemma \ref{lem:mean_dphi0}, and the last line uses part (a) of Lemma \ref{lem:assumption_a5}.
Next, we study
\[
    \EE \left(\phi_1(\hbbeta_{/1}) - \EE (\phi_1(\hbbeta_{/1})|\cD_{/1})\right)
    \cdot \left(\phi_{2}(\hbbeta_{/2}) - \EE (\phi_{2}(\hbbeta_{/2})|\cD_{/2})\right).
    \label{eq:v1_cross_term}\numberthis
\]
Define $ \hbbeta_{/12} :=  \underset{\bbeta \in \bTheta}{\argmin}  \Bigl \{   \sum_{j\geq 3}  \ell ( y_j,   \bx_j^\top \bbeta ) + \lambda r(\bbeta)  \Bigr \}$. By the mean-value theorem, there exists a random variable $t\in [0,1]$ such that
\[
    \phi_1(\hbbeta_{/1}) = \phi_1(\hbbeta_{/12}) + \dphi_1(t\hbbeta_{/1} + (1-t)\hbbeta_{/12}) \bx_1^\top (\hbbeta_{/1}-\hbbeta_{/12}).
\]
Hence,
\begin{align*}
    &~\EE [\phi_1(\hbbeta_{/1})|\cD_{/1}]\\
    =&~ \EE [\phi_1(\hbbeta_{/12})|\cD_{/1}] 
      + \EE [\dphi_1(t\hbbeta_{/1} + (1-t)\hbbeta_{/12}) \bx_1^\top (\hbbeta_{/1}-\hbbeta_{/12})|\cD_{/1}]\\
    =&~ \EE [\phi_0(\hbbeta_{/12})|\cD_{/12}] 
      + \EE [\dphi_0(t\hbbeta_{/1} + (1-t)\hbbeta_{/12})\bx_0^\top|\cD_{/1}](\hbbeta_{/1}-\hbbeta_{/12}).
\end{align*}
Similarly, we have
\[
    \phi_{2}(\hbbeta_{/2}) = \phi_{2}(\hbbeta_{/12}) + \dphi_{2}(t\hbbeta_{/2} + (1-t)\hbbeta_{/12}) \bx_{2}^\top (\hbbeta_{/2}-\hbbeta_{/12}). 
\]
and
\[
    \EE[\phi_{2}(\hbbeta_{/2})|\cD_{/2}] 
    = \EE [\phi_0(\hbbeta_{/12})|\cD_{/12}] 
      + \EE [\dphi_0(t\hbbeta_{/2} + (1-t)\hbbeta_{/12})\bx_0^\top|\cD_{/2}](\hbbeta_{/2}-\hbbeta_{/12})
\]
Then \eqref{eq:v1_cross_term} can be decomposed into four terms:
\begin{align*}\label{eq:decABCD}
    &\EE \left(\phi_1(\hbbeta_{/1}) - \EE (\phi_1(\hbbeta_{/1})|\cD_{/1})\right)
    \cdot \left(\phi_{2}(\hbbeta_{/2}) - \EE (\phi_{2}(\hbbeta_{/2})|\cD_{/2})\right)\\
    =& A_1 + B_1 + C_1 + D_1,\numberthis
\end{align*}
where $A_1, B_1, C_1,$ and $D_1$ are defined as 
\begin{align}
   A_1&:=~
    \EE \left( \phi_1(\hbbeta_{/12})  - \EE [ \phi_0(\hbbeta_{/12})  | \cD_{/12}] \right)\left(\phi_{2}( \hbbeta_{/12})  - \EE [ \phi_0( \hbbeta_{/12}) |  \cD_{/12}] \right) \nonumber \\
        B_1 &:=
    \EE \Big[ \left ( \phi_1(\hbbeta_{/12})  - \EE\bigl[ \phi_0(  \hbbeta_{/12})   |\cD_{/12} \bigr] \right ) \nonumber \\
    &\quad \times
    \Big ( \dphi_{2}(\hbbeta_{/2} + (1-t)\hbbeta_{/12}) \bx_{2}^\top (\hbbeta_{/2}-\hbbeta_{/12})- \EE [  \dphi_{2}(\hbbeta_{/2} + (1-t)\hbbeta_{/12}) \bx_{2}^\top (\hbbeta_{/2}-\hbbeta_{/12}) |\cD_{/2}  ]\Big) \Big] \nonumber \\
       C_1
    &:= \EE \Big [ \left ( \phi_{2}(\hbbeta_{/12})  - \EE [ \phi_0(  \hbbeta_{/12})  | \cD_{/12}] \right ) \nonumber \\
    &\quad \times
    \Big ( \dphi_1(\hbbeta_{/1} + (1-t) \hbbeta_{/12}) \bx_1^\top (t\hbbeta_{/1}-\hbbeta_{/12})- \EE \left [  \dphi_1(t \hbbeta_{/1} + (1-t)\hbbeta_{/12}) \bx_1^\top (\hbbeta_{/1}-\hbbeta_{/12}) |  \cD_{/1}  \right]\Big) \Big] \nonumber \\
        D_1 &:= 
    \EE \Biggl\{
      \biggl( \dphi_1(t\hbbeta_{/1} + (1-t)\hbbeta_{/12}) \bx_1^\top (\hbbeta_{/1}-\hbbeta_{/12}) - \EE \bigl[  \dphi_1(t\hbbeta_{/1} + (1-t) \hbbeta_{/12}) \bx_1^\top (\hbbeta_{/1}-\hbbeta_{/12}) |  \cD_{/1}  \bigr]\biggr) \nonumber \\
    &\quad \times \biggl( \dphi_{2}(\hbbeta_{/2} + (1-t)\hbbeta_{/12}) \bx_{2}^\top (\hbbeta_{/2}-\hbbeta_{/12})- \EE \bigl [  \dphi_{2}(\hbbeta_{/2} + (1-t)\hbbeta_{/12}) \bx_{2}^\top (\hbbeta_{/2}-\hbbeta_{/12}) |  \cD_{/2}  \bigr]\biggr) 
    \Biggr\}. \nonumber 
\end{align}
We bound each of these four terms below. For $A_1$ we have

\begin{align*}\label{eq:A1}
    A_1:=&~
    \EE \left( \phi_1(\hbbeta_{/12})  - \EE [ \phi_0(\hbbeta_{/12})  | \cD_{/12}] \right)\left(\phi_{2}( \hbbeta_{/12})  - \EE [ \phi_0( \hbbeta_{/12}) |  \cD_{/12}] \right) \\
    =&~
    \EE\left[ \EE\left[  \left( \phi_1(\hbbeta_{/12})  - \EE [ \phi_0(\hbbeta_{/12})  | \cD_{/12}] \right)\left.\left(\phi_{2}( \hbbeta_{/12})  - \EE [ \phi_0( \hbbeta_{/12}) |  \cD_{/12}] \right)  \right|\cD_{/12}\right] \right]\\
    =&~0.\numberthis
\end{align*}
where the last equality is correct, because given $\cD_{/12}$, $\phi_1(\hbbeta_{/12})$ and $\phi_{2}(\hbbeta_{/12})$ are conditionally independent, and  hence the inner expectation can be taken to each term separately. Similarly,
\begin{align*}\label{eq:B1}
    B_1 &:=
    \EE \Big[ \left ( \phi_1(\hbbeta_{/12})  - \EE\bigl[ \phi_0(  \hbbeta_{/12})   |\cD_{/12} \bigr] \right ) \\
    &\quad \times
    \Big ( \dphi_{2}(\hbbeta_{/2} + (1-t)\hbbeta_{/12}) \bx_{2}^\top (\hbbeta_{/2}-\hbbeta_{/12})- \EE [  \dphi_{2}(\hbbeta_{/2} + (1-t)\hbbeta_{/12}) \bx_{2}^\top (\hbbeta_{/2}-\hbbeta_{/12}) |\cD_{/2}  ]\Big) \Big]\\
    &=\EE\Big\{\EE\Big[  
     \left ( \phi_1(\hbbeta_{/12})  - \EE\bigl[ \phi_0(  \hbbeta_{/12})   |\cD_{/12} \bigr] \right ) \\
    &\quad \times
    \left.\Big ( \dphi_{2}(\hbbeta_{/2} + (1-t)\hbbeta_{/12}) \bx_{2}^\top (\hbbeta_{/2}-\hbbeta_{/12})- \EE [  \dphi_{2}(\hbbeta_{/2} + (1-t)\hbbeta_{/12}) \bx_{2}^\top (\hbbeta_{/2}-\hbbeta_{/12}) |\cD_{/2}  ]\Big)
    \right| \cD_{/2}\Big]\Big\}
    \\
    &=0.\numberthis
\end{align*}
Again, the last equality is correct, because given $\cD_{/2}$, $\phi_1(\hbbeta_{/12})$ and $\dphi_{2}(\hbbeta_{/2} + (1-t)\hbbeta_{/12}) \bx_{2}$ are conditionally independent. Similarly, 
\begin{align*}\label{eq:C1}
    C_1
    &:= \EE \Big [ \left ( \phi_{2}(\hbbeta_{/12})  - \EE [ \phi_0(  \hbbeta_{/12})  | \cD_{/12}] \right ) \\
    &\quad \times
    \Big ( \dphi_1(\hbbeta_{/1} + (1-t) \hbbeta_{/12}) \bx_1^\top (t\hbbeta_{/1}-\hbbeta_{/12})- \EE \left [  \dphi_1(t \hbbeta_{/1} + (1-t)\hbbeta_{/12}) \bx_1^\top (\hbbeta_{/1}-\hbbeta_{/12}) |  \cD_{/1}  \right]\Big) \Big] \\
    &= 0.\numberthis
\end{align*}
Finally,
\begin{align*}
    D_1 &:= 
    \EE \Biggl\{
      \biggl( \dphi_1(t\hbbeta_{/1} + (1-t)\hbbeta_{/12}) \bx_1^\top (\hbbeta_{/1}-\hbbeta_{/12}) - \EE \bigl[  \dphi_1(t\hbbeta_{/1} + (1-t) \hbbeta_{/12}) \bx_1^\top (\hbbeta_{/1}-\hbbeta_{/12}) |  \cD_{/1}  \bigr]\biggr)\\
    &\quad \times \biggl( \dphi_{2}(\hbbeta_{/2} + (1-t)\hbbeta_{/12}) \bx_{2}^\top (\hbbeta_{/2}-\hbbeta_{/12})- \EE \bigl [  \dphi_{2}(\hbbeta_{/2} + (1-t)\hbbeta_{/12}) \bx_{2}^\top (\hbbeta_{/2}-\hbbeta_{/12}) |  \cD_{/2}  \bigr]\biggr)
    \Biggr\}\\
    &\overset{(a)}{\leq }\EE \var \left[ \dphi_1(t\hbbeta_{/1} + (1-t) \hbbeta_{/12}) \bx_1^\top (\hbbeta_{/1}-\hbbeta_{/12})  |  \cD_{/1}  \right]  \\
    &\leq \EE \left[  \EE \left[\dphi_1^2(t\hbbeta_{/1} + (1-t)\hbbeta_{/12}) [\bx_1^\top (\hbbeta_{/1}-\hbbeta_{/12})]^2  |  \cD_{/1} \right]  \right]  \\
    &\overset{(b)}{\leq } \EE \left[ \sqrt{\EE[\dphi_1^4(\bxi_1)|\cD_{/1}]}\sqrt{\EE\left[[\bx_1^\top(\bbeta_{/1}-\bbeta_{/12})]^4|\cD_{/1}\right]} \right] \\
    &\overset{(c)}{\leq } \EE\left[ C_{\phi}^2\left(1+\left( \frac1p\|\bxi\|^2 \right)^s\right)\cdot \frac{\sqrt{3}C_X}{p}\| \hbbeta_{/1}-\hbbeta_{/12} \|^2 \right]
    \\
    &\overset{(d)}{\leq } \frac{\sqrt{3}C_{\phi}^2 C_X}{p(2\lambda\eta\wedge 1)^2}\EE\left[ \left(1+\left( p^{-1}\|\bxi\|^2 \right)^s\right) |\dl_2(\hbbeta_{/12})|^2\|\bx_2\|^2 \right]
    \\
    &\overset{(e)}{\leq } \frac{\sqrt{3}C_{\phi}^2 C_X}{p(2\lambda\eta\wedge 1)^2}\sqrt{ \EE  \left(1+\left( p^{-1}\|\bxi\|^2 \right)^s\right)^2}\cdot\left(\EE |\dl_2(\hbbeta_{/12})|^8 \right)^{\frac14} \left(\EE\|\bx_2\|^8\right)^{\frac14} 
    \\
    &\overset{(f)}{\leq } \frac{\sqrt{3}\gamma_0 C_{\phi}^2 C_X}{n(2\lambda\eta\wedge 1)^2}
    \sqrt{2+2C_{\beta}(2s)}C_{\ell}^{\frac14} (24C_X^4)^{\frac14}
    \\
    &:=\frac{C_{v1,2}}{n}
    \label{eq:v1_d1_last_step}\numberthis.
\end{align*}

where inequality (a) uses Cauchy Schwarz inequality and symmetry between $\dphi_1$ and $\dphi_2$, (b) uses Cauchy Schwarz again, (c) uses Lemma \ref{lem:mean_dphi0} and the fact that, conditioned on $\cD$, $\bx_1^\top(\hbbeta_{/1}-\hbbeta_{/12})\sim N(0,\hbbeta_{/1}-\hbbeta_{/12}^\top\bSigma\hbbeta_{/1}-\hbbeta_{/12})$. Inequality (d) uses Lemma \ref{lem:beta_lo_error}. Inequality (e) uses Cauchy Schwarz twice. Finally inequality (f) uses Lemma \ref{lem:assumption_a5} and Lemma \ref{lem:sum_xi_conc}. Plugging in the results of equations~\eqref{eq:A1}-\eqref{eq:v1_d1_last_step} into \eqref{eq:decABCD}, and then using \eqref{eq:v1_first_term}, the bound in \eqref{eq:V1bd} boils down to
\begin{align*}
    \EE V_1^2 
    \leq \dfrac{C_{v1,1}+C_{v1,2} }{n} := \frac{C_{v1}}{n}
\end{align*}
where, after some simplification:
\[
    C_{v1} = 2C_{\phi}^2\left(1+ C_{\beta}(s)\right) + \frac{\sqrt{6}(24)^{\frac14} C_{\phi}^2 C_X^2C_{\ell}^{\frac14}\gamma_0}{(2\lambda\eta\wedge 1)^2}
    \sqrt{1+C_{\beta}(2s)}. \label{eq:constant_v2}\numberthis
\]
This finishes the proof of Lemma~\ref{lem:v1}.
\qed

\end{document}